\documentclass[12pt, a4paper]{article}
\usepackage[latin1]{inputenc}
\usepackage[english]{babel}
\usepackage{amsmath,verbatim,amsthm,mathrsfs,stmaryrd}
\usepackage{hyperref}
\usepackage{indentfirst}
\usepackage{amstext}
\usepackage{enumerate}
\usepackage{amsfonts} 
\usepackage{textcomp}
\usepackage{amssymb}
\usepackage{setspace}
\usepackage{color}
\usepackage[normalem]{ulem}
\usepackage[dvipsnames]{xcolor}
\usepackage{graphicx}
\graphicspath{ {./images/} }
\usepackage[all]{xy}
\usepackage[noabbrev, capitalise, nameinlink, nosort]{cleveref}
\usepackage[mathscr]{euscript}

\newcommand {\R}{\mathbb{R}}
\newcommand {\N}{\mathbb{N}}

\newcommand{\bgln}{\begin{eqnarray}} 
\newcommand{\egln}{\end{eqnarray}}
\newcommand{\bgl}{\begin{equation}} 
\newcommand{\egl}{\end{equation}}

\newcommand{\m}[2]{\mathcal {#1}_ {#2}}

\usepackage{geometry}
\newcommand{\diam}{\operatorname{{diam}}}

\newcommand{\dist}{\operatorname{{dist}}}

\newcommand{\Dom}{\operatorname{Dom}}

\newtheorem{teorema}{theorem}[section]
\newtheorem{lemma}[teorema]{Lemma}
\newtheorem{corollary}[teorema]{Corollary}
\newtheorem{definition}[teorema]{Definition}
\newtheorem{proposition}[teorema]{Proposition}

\newtheorem{example}[teorema]{Example}
\theoremstyle{remark}
\theoremstyle{definition}
\newtheorem{remark}[teorema]{Remark}
\newtheorem{theorem}[teorema]{Theorem}

\title{Paterson compactifications, inverse limits and shadowing for Deaconu--Renault systems}

\author{
Daniel Gon\c{c}alves$^{1}$\thanks{The first author was partially supported by Conselho Nacional de Desenvolvimento Cient\'ifico e Tecnol\'ogico (CNPq) - Brazil, and Funda\c{c}\~ao de Amparo \`a Pesquisa e Inova\c{c}\~ao do Estado de Santa Catarina (FAPESC).}
\and
Danilo Royer$^{1}$
\and
Felipe Augusto Tasca$^{1,2}$\thanks{The third author was partially supported by Funda\c{c}\~ao de Amparo \`a Pesquisa e Inova\c{c}\~ao do Estado de Santa Catarina (FAPESC).}
}

\date{}

\begin{document}

\maketitle

\begin{abstract}
We develop a new metric and inverse-limit framework for Deaconu--Renault systems arising from local homeomorphisms between open subsets of locally compact zero-di\-men\-si\-o\-nal spaces. Our starting point is the Paterson-type compactification of infinite product spaces, which underlies several symbolic and groupoid models, including one-sided shifts over infinite alphabets and path spaces of graphs and higher-rank graphs. We construct an explicit compatible ultrametric on this compactification and give a concrete description of its generalized cylinder topology and convergence.

Within this framework, we introduce an inverse-limit-type space naturally associated to a Deaconu--Renault system. In contrast with the classical compact theory, the correct inverse-limit object must incorporate not only infinite backward orbits but also finite configurations arising as limits of such orbits. This produces a canonical shift system extending the original dynamics.

We then study shadowing in the ultrametric setting. For spaces admitting tame defining sequences, we characterize shadowing in terms of the defining partitions, extending the topological description of shadowing from compact zero-dimensional dynamics to the locally compact setting relevant for Deaconu--Renault systems. As a main application, we prove a transfer theorem showing, under a separation property expressed through uniformly contracting inverse branches, that shadowing for the inverse-limit Deaconu--Renault system is equivalent to shadowing for the compactified base system. This provides a noncompact inverse-limit shadowing theory for a broad class of partially defined local-homeomorphism dynamics.
\end{abstract}

\vspace{4mm}



\vspace{3mm}

\noindent
$^{1}$Departamento de Matem\'atica, Federal University of Santa Catarina,
88040-970, Florian\'opolis, SC, Brazil.

\noindent
$^{2}$Federal Institute of Paran\'a,
84600-000, Uni\~ao da Vit\'oria, PR, Brazil.

\noindent
\textit{E-mail addresses:}
\href{mailto:daemig@gmail.com}{daemig@gmail.com},
\href{mailto:danilo.royer@ufsc.br}{danilo.royer@ufsc.br},
\href{mailto:tasca.felipe@gmail.com}{tasca.felipe@gmail.com}.

\vspace{4mm}

\section{Introduction}

Deaconu--Renault systems form a natural topological framework for dynamics generated by a
local homeomorphism between open subsets of a locally compact Hausdorff space. They occur
throughout symbolic dynamics, topological graph dynamics, and the groupoid approach to
operator algebras. In particular, one-sided shift spaces over infinite alphabets, boundary-path
spaces of directed graphs, and path spaces of higher-rank graphs all exhibit the same basic
feature: the natural state space is not a compact product space in the classical sense, but a
locally compact zero-dimensional space together with a partially defined shift-type map
\cite{deaconu1995,OTW14,webster2014,webster2011}. This setting is rich enough to capture
important examples, but it falls outside the standard compact theory, in which inverse limits
and shadowing are usually formulated.

The purpose of this paper is to develop a metric and inverse-limit framework for this
noncompact zero-dimensional setting, and to use it to obtain new results on shadowing for
Deaconu--Renault systems. The point is not merely to transport known compact arguments to a
slightly more general context. Rather, the objects that arise here force new constructions and
new proofs. In the classical compact theory, one starts with a globally defined map on a metric
space and forms its inverse limit inside a product space whose topology and metric are already
available. In the present setting, none of these ingredients is automatic: the ambient compact
space must first be built, its topology must be made explicit, the correct inverse-limit object
must accommodate finite as well as infinite backward trajectories, and the shadowing problem
must be reformulated so that it interacts correctly with partial domains and local inverse
branches. These difficulties are structural, and resolving them leads to a new theory.

Our starting point is the compactification introduced by Paterson and Welch for countable
products of locally compact spaces, a construction that also underlies the OTW model for
one-sided shifts over infinite alphabets and related path-space constructions in graph dynamics
\cite{GoncalvesRoyerTasca2024, OTW14,webster2011,webster2014}. It has long been known that these spaces are
metrizable in the countable case, but for the purposes of dynamics one needs considerably
more than abstract metrizability. One needs an explicit metric, an effective description of
convergence, and a basis adapted to the symbolic and inverse-limit structure. The first part of
the paper supplies exactly this. We construct a concrete compatible ultrametric on the
Paterson-type compactification and describe its generalized cylinder topology in a form suited
to dynamical arguments. This gives a usable metric model for a class of spaces that has become
increasingly important in symbolic and groupoid dynamics.

The next step is to identify the correct inverse-limit object in this context. If $f:\Dom(f)\to \Im(f)$ is a local homeomorphism between open subsets of a locally compact Hausdorff space, then the
classical inverse limit is no longer the right object: finite backward paths and limits of infinite
backward paths cannot simply be discarded. We introduce an inverse-limit-type space
$\widetilde X$ that incorporates both infinite backward orbits and the finite configurations that
arise as their limits in the Paterson compactification. This construction comes equipped with
natural shift maps and yields a Deaconu--Renault system canonically associated to the original
dynamics. In this way, the paper provides a genuine inverse-limit formalism for partially
defined zero-dimensional systems. To our knowledge, this formulation does not appear in the
existing literature, even for basic examples motivated by symbolic dynamics over infinite
alphabets.

The second theme of the paper is shadowing. In compact metric dynamics, shadowing admits
powerful descriptions in terms of symbolic models, inverse limits, and clopen partitions. In
particular, Good and Meddaugh showed that shifts of finite type play a fundamental role in the
theory of shadowing on compact totally disconnected spaces, while Darji, Gon\c calves, and
Sobottka developed a characterization of shadowing in ultrametric spaces in terms of defining
sequences \cite{GoodMeddaugh2020, DGS}. One of our main results
extends this circle of ideas to the setting relevant for Deaconu--Renault systems. In
Section~\ref{nintendoDS} we characterize shadowing in terms of tame defining sequences for
the ultrametric spaces that arise from our compactification procedure. This yields a
topological-ultrametric criterion for shadowing beyond the compact framework of the existing
theory. 

The inverse-limit construction and the explicit ultrametric are then used to obtain the main
dynamical application of the paper: a transfer theorem for shadowing between the compactified
base system and the inverse-limit Deaconu--Renault system. Under a separation property
formulated in terms of a compatible ultrametric on the one-point compactification and
uniformly contracting local inverse branches, we prove that shadowing for the inverse-limit
system $(E,\alpha_f)$ is equivalent to shadowing for the compactified base system
$(D^\infty,f)$. This generalizes the classical transfer principle for inverse limits of
homeomorphisms and continuous surjections on compact spaces \cite{chenli1993}. In the Deaconu--Renault setting one must
control partial domains, length-one words, the distinguished zero word, and the interaction
between finite and infinite backward trajectories. The separation property introduced here is
designed precisely to make this possible. It provides a workable mechanism for converting
pseudo-orbits in the base system into exact inverse-limit trajectories and conversely, in a
setting where standard compact partition arguments do not apply directly.

From a broader perspective, the paper develops a new toolkit for noncompact zero-di\-men\-si\-o\-nal
dynamics. Its contribution is not only that it proves several results in parallel; it is that it
creates a framework in which these results become accessible at all. The explicit metric on the
Paterson compactification, the inverse-limit construction for Deaconu--Renault systems, the
characterization of shadowing via tame defining sequences, and the transfer theorem between
base and inverse-limit dynamics interact in an
essential way. The metric makes the topology effective; the inverse-limit construction supplies
the correct phase space upstairs; the defining-sequence criterion makes shadowing testable in
the ultrametric regime; and the transfer theorem links the two levels of dynamics. Together, they open a route to studying symbolic and groupoid models of noncompact local-homeomorphism
dynamics with tools that, until now, were available only in more restrictive compact settings.

The paper is organized as follows. In Section~\ref{W0 section} we revisit the Paterson-type
construction, describe its topology by generalized cylinders, and construct a compatible
ultrametric. In Section~\ref{tornado} we introduce the inverse-limit space $\widetilde X$ and
the associated shift maps, and we show that these yield the Deaconu--Renault system attached
to the inverse-limit construction. In Section~\ref{nintendoDS} we prove the characterization of
shadowing in terms of tame defining sequences in ultrametric spaces. Finally, in
Section~\ref{shadowthis} we establish the transfer theorem between the compactified base
system and the inverse-limit system, and we illustrate the theory with examples arising from
symbolic dynamics.

\section{Metric on infinite product spaces}\label{W0 section}

Throughout the paper, we assume that $X$ is an infinite locally compact Hausdorff space
that admits a countable basis of compact open sets closed under finite intersections,
denoted by $\mathcal{B}$.

In \cite{pat}, the authors define a topology on infinite product spaces of locally compact
spaces so that the resulting space is locally compact. They refer to this topology as the
topology of pointwise convergence. Under their assumptions, the resulting space is
metrizable. In this section, we recall the construction from \cite{pat}, describe a basis for
the topology (consisting of compact open sets), provide an alternative characterization of
convergence of sequences, and exhibit a compatible metric, which turns out to be an
ultrametric.

We begin by recalling the definition of compactified infinite product space as in
\cite{pat}. For each $n \geq 0$, let $
X^n = \underbrace{X \times \cdots \times X}_{n \text{ times}}$, with $X^0 = \{\vec{0}\}$. Define
\[
W_0 = X^{\mathbb{N}} \,\cup\, \Big( \bigcup_{n=0}^\infty X^n \Big).
\]

Let $X_\infty := X \cup \{\infty\}$ be the one-point compactification of $X$. Then $X_\infty$
is a compact Hausdorff space. By Tychonoff's theorem, the countably infinite product $
A := X_\infty \times X_\infty \times X_\infty \times \cdots $
is also a compact Hausdorff space when endowed with the product topology.

Define a function $Q: A \to W_0$ by
\begin{equation}
Q(x_1 x_2 x_3 \ldots)
=
\begin{cases}
x_1 x_2 x_3 \ldots & \text{if } x_i \neq \infty \text{ for all } i,\\[4pt]
x_1 x_2 \ldots x_n & \text{if } x_{n+1} = \infty \text{ for some } n \text{ and } x_i \neq \infty \text{ for } 1 \le i \le n,\\[4pt]
\vec{0} & \text{if } x_1 = \infty.
\end{cases}
\label{specify}
\end{equation}

Observe that $Q$ is surjective. We equip $W_0$ with the quotient topology induced from $A$
via the map $Q$, that is, a set $U \subseteq W_0$ is open if and only if $Q^{-1}(U) \subseteq
A$ is open. Since $Q$ is continuous and $A$ is compact, it follows that $W_0$ is compact.
Define the length map $\ell : W_0 \to \mathbb{N} \cup \{\infty\}$ by $\ell(x) = n$ if $x \in X^n$ for some $n \ge 0$ and $\ell(x)=\infty$ if $x\in X^{\mathbb{N}}$.
Note that $\vec{0}$ is the unique element of $W_0$ with length $0$.

A basis for the topology defined above is not provided in \cite{pat}. We remedy this below, by constructing a basis formed by certain generalized cylinder sets, which we define below.

\begin{definition}\label{gcyl}
Let $B_1, \ldots, B_k$ be
elements of $\mathcal{B}$, $k\in \N$, and let $K \subseteq X$ be a compact open set. The subsets of $W_0$ defined by
\begin{align*}
Z[B_1 \times \cdots \times B_k, K]
  &= \left\{\, x \in W_0 :
      \ell(x) \ge k,\ 
      (x_1, \ldots, x_k) \in B_1 \times \cdots \times B_k,\ 
      x_{k+1} \notin K
     \right\}, \\
C[K]
  &= \left\{\, x \in W_0 :
      x = \vec{0}
      \ \text{or}\
      x_1 \notin K
     \right\},
\end{align*}
are called generalized cylinders.
\end{definition}

\begin{remark}
Note that $C[K]$ and $Z[K,\emptyset]$ are distinct sets. Moreover, whenever $K \in \mathcal{B}$,
we have $
C[K] = W_0 \setminus Z[K,\emptyset].$
\end{remark}

\begin{proposition}
The collection of generalized cylinders  $\{Z[B_1 \times \cdots \times B_k, K], C[K]\}$, where $k\in \N$, $B_i\in \mathcal B$, and $K$ is a compact open subset of $X$ forms a basis for a topology on $W_0$.
\end{proposition}

\begin{proof}
   It is clear that $W_0 = C[\emptyset]$.
    
We aim to show that, given two cylinder sets $Z_1$ and $Z_2$ in $W_0$, then each element of $Z_1\cap Z_2$ belongs to some cylinder set $D$ contained in $Z_1 \cap Z_2$. To this end, we consider three distinct cases.

    \begin{itemize}
        \item For $Z_1=Z[B_1 \times \ldots \times B_k, K]$ and $Z_2=Z[C_1 \times \ldots \times C_n, L]$ we have two subcases:
        \begin{itemize}
            \item Case $k=n$. In this case, define for each $i \in \{1,\ldots, k\}$, the set $D_i:=B_i\cap C_i\in \mathcal{B}$. Furthermore, let $M=K\cup L$, which is a compact open set. Now let $D=Z[D_1\times \ldots \times D_k,M]$ and notice that $D= Z_1 \cap Z_2$. Therefore, each element of $Z_1\cap Z_2$ belongs to $D$.
        
            \item Case $k<n$. Let $x\in Z_1\cap Z_2$. For each $i\in \{1, \ldots , k\}$, define $D_i=B_i\cap C_i$. Now, we see that $C_{k+1}\cap (X\setminus K)$ is open in $X$, and then there exists $D_{k+1}\in \mathcal{B}$ such that $x_{k+1}\in D_{k+1}\subseteq C_{k+1}\cap (X\setminus K)$. Now define  
            $D=Z[D_1\times\ldots\times D_k \times D_{k+1}, C_{k+2}\times \ldots\times C_n, L]$, and notice that $x\in D\subseteq Z_1\cap Z_2$.
            \end{itemize}

        \item For $Z_1=Z[B_1 \times \ldots \times B_k, K]$ and $Z_2=C[L]$, observe that $B_1 \cap (X \setminus L)$ is open and, therefore, there exists $D_1 \in \mathcal{B}$ such that $x_1 \in D_1 \subseteq B_1 \cap (X \setminus L)$. Therefore, $x \in D := Z[D_1 \times B_2 \times \ldots \times B_k, K] \subseteq Z_1 \cap Z_2$.

        \item For $Z_1=C[K]$ and $Z_2=C[L]$, with $K,L$ compact open sets, notice that $Z_1\cap Z_2=Z[K\cap L]$. In this case, let $D=Z_1\cap Z_2$.
    \end{itemize}
\end{proof}

From now on, we also consider the generalized cylinder set topology on $W_0$, that is, the
topology generated by the generalized cylinders $Z[P,K]$ with $P \in \mathcal{B}^k$ for
some $k \ge 0$ and $K \subseteq X$ a compact open set. Our goal is to show that this
generalized cylinder set topology coincides with the quotient topology. The first step in
this direction is the following:

\begin{lemma}\label{barcoleve}
Let $B_1, B_2, \ldots, B_k$ be elements of $\mathcal{B}$ and let $K \subseteq X$ be a compact
open set. Then the generalized cylinder sets $Z[B_1 \times \cdots \times B_k, K]$ and $C[K]$
are compact open subsets of $W_0$ in the quotient topology.
\end{lemma}

\begin{proof}
We now show that $Z[B_1 \times \cdots \times B_k, K]$ and $C[K]$ are open in $W_0$ with the
quotient topology. Consider their preimages under $Q$:
\[
Q^{-1}\big(Z[B_1 \times \cdots \times B_k, K]\big)
= B_1 \times \cdots \times B_k \times (X_\infty \setminus K)
\times X_\infty \times X_\infty \times \cdots,
\]
which is open in $A$ since each $B_i$ is open in $X_\infty$ and $X_\infty \setminus K$ is open
in $X_\infty$ (because $K$ is compact $X$).

Similarly,
\[
Q^{-1}(C[K])
= (X_\infty \setminus K) \times X_\infty \times X_\infty \times \cdots,
\]
which is open in $A$ as well. Therefore, both $Z[B_1 \times \cdots \times B_k, K]$ and $C[K]$
are open in $W_0$.

\smallskip
Next, we show that $Z[B_1 \times \cdots \times B_k, K]$ is closed in $W_0$ (with the quotient
topology). Let $y \notin Z[B_1 \times \cdots \times B_k, K]$. We consider two cases: 
$\ell(y) < k$ and $\ell(y) \ge k$.

\smallskip

\noindent
\textbf{Case 1:} $\ell(y) < k$. 
If $\ell(y)=0$, then $y=\vec{0}$ and $y \in C[B_1]$, which is open in $W_0$, and clearly
$C[B_1] \cap Z[B_1 \times \cdots \times B_k, K] = \emptyset$.

If $\ell(y) > 0$, write $y = y_1 \cdots y_n$ with $n < k$. For each $i = 1, \ldots, n$, choose
$C_i \in \mathcal{B}$ such that $y_i \in C_i$. Then $
y \in Z[C_1 \times \cdots \times C_n, B_{n+1}],$ and $
Z[C_1 \times \cdots \times C_n, B_{n+1}]
\cap
Z[B_1 \times \cdots \times B_k, K]
=
\emptyset.$

\smallskip

\noindent
\textbf{Case 2:} $\ell(y) \ge k$.
Write $y = y_1 \cdots y_n$ with $n \ge k$. Since $y \notin Z[B_1 \times \cdots \times B_k, K]$,
either:

(i) $y_i \notin B_i$ for some $1 \le i \le k$, or

(ii) $y_i \in B_i$ for all $1 \le i \le k$, and $y_{k+1} \in K$.

If (i) holds, let $j = \min\{ i : y_i \notin B_i \}$. Choose $C_j \in \mathcal{B}$ such that
$y_j \in C_j \subseteq X \setminus B_j$. Then $y \in Z[B_1 \times \cdots \times B_{j-1} \times C_j, \emptyset]$,
which is an open subset of $W_0$ disjoint from $Z[B_1 \times \cdots \times B_k, K]$.

If (ii) holds, then $y_{k+1} \in K$. Since $K$ is open, choose $B_{k+1} \in \mathcal{B}$ such
that $y_{k+1} \in B_{k+1} \subseteq K$. Then $
Z[B_1 \times \cdots \times B_{k+1}, \emptyset]
\cap
Z[B_1 \times \cdots \times B_k, K]
=
\emptyset.$

\smallskip

In all cases, we have exhibited an open neighborhood of $y$ disjoint from
$Z[B_1 \times \cdots \times B_k, K]$, proving that the latter is closed.

We now show that $C[K]$ is closed. Let $y \notin C[K]$. Since $\vec{0} \in C[K]$,
we must have $\ell(y) \ge 1$, so write $y = y_1 y_2 \cdots$. Because $y \notin C[K]$, we have
$y_1 \in K$. Since $K$ is open in $X$, there exists $C_1 \in \mathcal{B}$ such that
$y_1 \in C_1 \subseteq K$. Then $
y \in Z[C_1, \emptyset]
\quad \text{and} \quad
Z[C_1, \emptyset] \cap C[K] = \emptyset.$
Thus $C[K]$ is closed in $W_0$.

\smallskip

We have now shown that all generalized cylinders are both open and closed in $W_0$. Since
$W_0$ is compact in the quotient topology, every such cylinder is also compact.
\end{proof}

To prove that the generalized cylinder topology coincides with the quotient topology on
$W_0$, and to define a compatible metric later, we introduce the following auxiliary map.

Let $\mathcal{P} = \bigcup_{n=1}^\infty \mathcal{B}^n$, and define
$\alpha : W_0 \to \{0,1\}^{\mathcal{P}}$ by
\begin{equation}
\alpha(x)(B_1 \times \cdots \times B_k)
=
\begin{cases}
1 & \text{if } x \in Z[B_1 \times \cdots \times B_k, \emptyset], \\[4pt]
0 & \text{otherwise}.
\end{cases}
\label{eq:alpha}
\end{equation}

The map $\alpha : W_0 \to \{0,1\}^{\mathcal{P}}$ is not surjective. For example, there is no
$x \in W_0$ such that
$\alpha(x)(B_1 \times \cdots \times B_k) = 1$
for every
$B_1 \times \cdots \times B_k \in \mathcal{P}$.
However, it is injective and continuous, as we show next.

\begin{proposition}\label{propalpha}
The function
$\alpha : W_0 \to \alpha(W_0) \subseteq \{0,1\}^{\mathcal{P}}$
is a homeomorphism (where $W_0$ is equipped with the quotient topology and
$\{0,1\}^{\mathcal{P}}$ with the product topology).
\end{proposition}

\begin{proof}
    We show that $\alpha$ is a homeomorphism onto its image. To prove that $\alpha$ is injective,
let $x,y \in W_0$ with $x \ne y$. We distinguish two cases.

\begin{itemize}
\item[\textbf{(i)}] $\ell(x) = \ell(y)$.  
Then there exists a minimal $k \ge 1$ such that $x_k \ne y_k$. Since $X$ is Hausdorff, we
can find open sets $U,V \subseteq X$ such that $x_k \in U$, $y_k \in V$, and
$U \cap V = \emptyset$. Because $\mathcal{B}$ is a basis of compact open sets, we may choose
$B_k \in \mathcal{B}$ such that $x_k \in B_k \subseteq U$, so in particular $y_k \notin B_k$.
For each $i = 1, \ldots, k-1$, choose $B_i \in \mathcal{B}$ such that $x_i \in B_i$.
Then
\[
x \in Z[B_1 \times \cdots \times B_k, \emptyset]
\quad \text{and} \quad
y \notin Z[B_1 \times \cdots \times B_k, \emptyset].
\]
Thus, $
\alpha(x)(B_1 \times \cdots \times B_k)
\ne
\alpha(y)(B_1 \times \cdots \times B_k)$,
and therefore $\alpha(x) \ne \alpha(y)$.

\item[\textbf{(ii)}] $\ell(x) \ne \ell(y)$.  
Without loss of generality, suppose $\ell(x) < \infty$ and $\ell(x) < \ell(y)$. Let $\ell(x) = k$.
Choose $B_1, \ldots, B_{k+1} \in \mathcal{B}$ such that $y \in
Z[B_1 \times \cdots \times B_{k+1}, \emptyset]$. Since $\ell(x) = k < k+1$, we have
$x \notin Z[B_1 \times \cdots \times B_{k+1}, \emptyset]$. Hence
\[
\alpha(x)(B_1 \times \cdots \times B_{k+1}) = 0 \ne 1
= \alpha(y)(B_1 \times \cdots \times B_{k+1}),
\]
and therefore $\alpha(x) \ne \alpha(y)$.
\end{itemize}

To prove that $\alpha$ is continuous, recall that $\{0,1\}^{\mathcal{P}}$ is equipped with
the product topology. Thus, it suffices to show that, for each
$B_1 \times \cdots \times B_k \in \mathcal{P}$, the composition
$\pi_{B_1 \times \cdots \times B_k} \circ \alpha : W_0 \to \{0,1\}$ is continuous, where
$\pi_{B_1 \times \cdots \times B_k}$ denotes the corresponding coordinate projection.
But
\[
(\pi_{B_1 \times \cdots \times B_k} \circ \alpha)^{-1}(1)
= Z[B_1 \times \cdots \times B_k, \emptyset]
\]
and
\[
(\pi_{B_1 \times \cdots \times B_k} \circ \alpha)^{-1}(0)
= W_0 \setminus Z[B_1 \times \cdots \times B_k, \emptyset],
\]
both of which are open subsets of $W_0$ by Proposition~\ref{barcoleve}. Thus each coordinate
map is continuous, and consequently $\alpha$ is continuous.

Since $\alpha$ is a continuous bijection from the compact Hausdorff space $W_0$ onto the
subspace $\alpha(W_0)$ of $\{0,1\}^{\mathcal{P}}$, it follows that $\alpha$ is a homeomorphism
onto its image.

\end{proof}

We now prove an auxiliary result concerning the map
$\alpha : W_0 \to \alpha(W_0)$.

\begin{lemma}\label{preimagcyl}
Let $F, G \subseteq \mathcal{P}$ be finite sets, and define the cylinder set
\[
C_{F,G} = \{\, w \in \alpha(W_0) : w(p) = 1 \ \text{for all } p \in F,\ 
w(q) = 0 \ \text{for all } q \in G \,\}.
\]
Then $\alpha^{-1}(C_{F,G})$ is open in the generalized cylinder set topology on $W_0$.
\end{lemma}

\begin{proof}
We claim that
\[
\alpha^{-1}(C_{F,G})
=
\left(\bigcap_{p \in F} Z[p,\emptyset] \right)
\cap
\left(\bigcap_{q \in G} Z[q,\emptyset]^c\right).
\tag{$\ast$}
\]
Indeed, for any $z \in W_0$ we have:
\begin{align*}
z \in \alpha^{-1}(C_{F,G})
&\iff \alpha(z) \in C_{F,G} \\
&\iff \alpha(z)(p)=1 \ \text{for all } p \in F
\ \text{and}\ 
\alpha(z)(q)=0 \ \text{for all } q \in G \\
&\iff z \in Z[p,\emptyset] \ \text{for all } p \in F
\ \text{and}\ 
z \notin Z[q,\emptyset] \ \text{for all } q \in G \\
&\iff 
z \in 
\left(\bigcap_{p \in F} Z[p,\emptyset] \right)
\cap
\left(\bigcap_{q \in G} Z[q,\emptyset]^c\right).
\end{align*}

Since each $Z[p,\emptyset]$ is open in $W_0$ by Lemma~\ref{barcoleve}, and $F$ and $G$
are finite sets, we conclude that $\alpha^{-1}(C_{F,G})$ is open in the generalized cylinder topology.
\end{proof}

With the above lemma, we can now prove that the topology of generalized cylinders coincides
with the quotient topology on $W_0$.

\begin{proposition}\label{same topologies}
The topology of generalized cylinders on $W_0$ coincides with the quotient topology induced
by the surjective map $Q : A \to W_0$.
\end{proposition}

\begin{proof}
We already proved in Lemma~\ref{barcoleve} that every open set in the generalized cylinder
topology is open in the quotient topology.

Conversely, let $U \subseteq W_0$ be open in the quotient topology, and let
$z \in U$. Then
$\alpha(z) \in \alpha(U)$, and $\alpha(U)$ is open in $\alpha(W_0)$ (with the subspace
topology inherited from $\{0,1\}^{\mathcal{P}}$). Hence, there exists a cylinder set
$C_{F,G}$ (as in Lemma~\ref{preimagcyl}) such that $
\alpha(z) \in C_{F,G} \subseteq \alpha(U).$ Therefore,
$
z \in \alpha^{-1}(C_{F,G}) \subseteq U$.
By Lemma~\ref{preimagcyl}, $\alpha^{-1}(C_{F,G})$ is open in the generalized cylinder set
topology. Since $z$ was arbitrary, this shows that $U$ is open in the generalized cylinder
topology.

Thus the two topologies coincide.
\end{proof}

We now describe neighborhood bases for the generalized cylinder set topology on $W_0$.

\begin{proposition}\label{lemabaseviz}
Let $x \in W_0$. Then:
\begin{itemize}
    \item If $\ell(x) = 0$ (so $x = \vec{0}$), then the sets of the form $C[K]$,
    where $K \subseteq X$ is compact and open, form a neighborhood basis at $x$.

    \item If $\ell(x) \ge 1$ (so $x = x_1 x_2 \ldots$), then the sets of the form
    $Z[B_1 \times \cdots \times B_k, K]$ (with $1 \le k \le \ell(x)$), where
    $x_i \in B_i \in \mathcal{B}$ for $i = 1,\ldots,k$ and
    $x_{k+1} \notin K$ whenever $\ell(x) > k$, form a neighborhood basis at $x$.
\end{itemize}
\end{proposition}

\begin{proof}
If $x = \vec{0}$, then the result follows immediately from the definition, since
the only generalized cylinders containing $\vec{0}$ are those of the form $C[K]$
with $K$ compact and open.

Now suppose $x = x_1x_2\ldots$ with $\ell(x) \ge 1$. Let $U \subseteq W_0$ be an open set with $x \in U$. Then there exists a basic open set $V$ such
that $x \in V \subseteq U$.

If $V = Z[B_1 \times \cdots \times B_k, K]$ for some
$B_1,\ldots,B_k \in \mathcal{B}$ and compact open $K \subseteq X$, then since $x \in V$ we
have $\ell(x) \ge k$, $x_i \in B_i$ for each $i \le k$, and if $\ell(x) > k$ then
$x_{k+1} \notin K$. Thus $V$ already has the desired form.

If instead $V = C[K]$ for some compact open $K \subseteq X$, then $x \neq \vec{0}$
implies $\ell(x) \ge 1$ and thus $x_1 \in X \setminus K$. As $X \setminus K$ is open, there
exists $B_1 \in \mathcal{B}$ compact and open such that
$x_1 \in B_1 \subseteq X \setminus K$. Then
$
x \in Z[B_1, \emptyset] \subseteq C[K] \subseteq U,$
and $Z[B_1, \emptyset]$ has the desired form. This completes the proof.
\end{proof}

\begin{remark}\label{ta meio confuso}
Let $(x^n)_{n \in \mathbb{N}} \subseteq W_0$ be a sequence converging to an element
$x \in W_0$. Write
$x^n = x_1^n x_2^n \cdots$ and $x = x_1 x_2 \cdots$.
In the next proposition, our goal is to describe this convergence in terms of the
coordinate sequences of the elements $x^n$.

However, the notation $(x_i^n)_{n \in \mathbb{N}}$ is not always well-defined, since it may
happen that $\ell(x^n) < i$ for some $n$, in which case the coordinate $x_i^n$ does not exist.
Thus, for a fixed $i$, we use the notation
\[
(x_i^n)_{n :\, \ell(x^n) \ge i},
\]
meaning the sequence (in $X$) consisting of the $i$-th coordinates of all those $x^n$ whose
length is at least $i$.
\end{remark}

As a consequence of the previous proposition, we now describe the convergence of sequences
in $W_0$. The notation introduced above makes this description precise, whereas in
\cite[Theorem~3.6]{pat} the corresponding discussion is somewhat more informal (for instance,
it does not take into account initial finite segments of the sequence).

\begin{corollary} \label{convergencia em W0} (cf.\ \cite[Theorem~3.6]{pat})
Let $(x^n)_{n\in \mathbb{N}}$ be a sequence in $W_0$, where each $x^n = x_1^n x_2^n \cdots$,
and let $x = x_1 x_2 \cdots \in W_0$.
\begin{enumerate}[(a)]
    \item Suppose that $\ell(x) = k < \infty$. Then $\lim_{n\to\infty} x^n = x$  if and only if:
    \begin{itemize}
    \item there exists $N \in \mathbb{N}$ such that $\ell(x^n) \ge k$ for all $n \ge N$,
    \item for each $i \in \{1,\ldots,k\}$, the sequence $(x_i^n)_{n :\, \ell(x^n) \ge i}$ converges to $x_i$ in $X$,
    \item the sequence $(x_{k+1}^n)_{n :\, \ell(x^n) \ge k+1}$ converges to $\infty$ in $X_\infty$,
    in case there are infinitely many $n$ with $\ell(x^n)\ge k+1$.
    \end{itemize}

    \item If $\ell(x) = \infty$, then $\lim_{n\to\infty} x^n = x$ if and only if:
    \[
    \ell(x^n) \to \infty
    \quad\text{and}\quad
    (x_i^n)_{n :\, \ell(x^n) \ge i} \to x_i \text{ in } X
    \text{ for each } i \in \mathbb{N}.
    \]
\end{enumerate}
\end{corollary}

\begin{proof}
We begin with the case $\ell(x) = k < \infty$.

Assume that $(x^n)$ converges to $x$. Fix
$i \in \{1,\ldots,k\}$ and let $U_i$ be an open neighborhood of $x_i$ in $X$. Choose
$B_i \in \mathcal{B}$ with $x_i \in B_i \subseteq U_i$, and for each
$j \in \{1,\ldots,k\}$, $j\neq i$, choose $B_j \in \mathcal{B}$ with $x_j \in B_j$. Let
$K \subseteq X$ be a compact open set. Then $x \in Z[B_1 \times \cdots \times B_k, K]$, which is an open
neighborhood of $x$ in $W_0$. Since $x^n \to x$, there exists $N \in \mathbb{N}$ such that
for all $n \ge N$ we have $x^n \in Z[B_1 \times \cdots \times B_k, K]$. In particular,
$\ell(x^n) \ge k$ and $x_i^n \in U_i$ for all $n \ge N$, so
$(x_i^n)_{n : \ell(x^n) \ge i} \to x_i$ in $X$.

Now assume there are infinitely many $n$ with $\ell(x^n) \ge k+1$. For those $n \ge N$ we have
$x_{k+1}^n \notin K$. Since $X_\infty \setminus K$ is open
and contains $\infty$, it follows that
$(x_{k+1}^n)_{n : \ell(x^n) \ge k+1} \to \infty$ in $X_\infty$.

Conversely, assume the three bullet conditions hold. Let $U$ be an open neighborhood of $x$
in $W_0$. By Proposition~\ref{lemabaseviz}, there exists a basic neighborhood
$Z[B_1 \times \cdots \times B_k, K]$ such that
$x \in Z[B_1 \times \cdots \times B_k, K] \subseteq U$. Since $(x_i^n)_{n :\, \ell(x^n)\ge i}$ converges to $x_i$ for each $i \in \{1,\ldots,k\}$ and
$\ell(x^n)\ge k$ for all $n \ge N$, there exists $N_1 \in \mathbb{N}$ such that
$x_i^n \in B_i$ for all $n \ge N_1$ and all $i \in \{1,\ldots,k\}$. Furthermore, since
$(x_{k+1}^n)_{n :\, \ell(x^n)\ge k+1}$ converges to $x_{k+1}$ (in case there are infinitely
many $n$ with $\ell(x^n)\ge k+1$), there exists $N_2 \in \mathbb{N}$ such that
$x_{k+1}^n \notin K$ for all $n \ge N_2$ for which $\ell(x^n)\ge k+1$. Thus, for all
$n \ge \max\{N_1,N_2\}$ we have
$x^n \in Z[B_1 \times \cdots \times B_k, K] \subseteq U$.
Hence $x^n \to x$ in $W_0$.

We now prove the case $\ell(x) = \infty$.

Assume first that $x^n \to x$. Fix $m \in \mathbb{N}$ and $i \in \{1,\ldots,m\}$, and let
$U_i \subseteq X$ be an open neighborhood of $x_i$. Choose $B_j \in \mathcal{B}$ for
$j = 1,\ldots,m$ with $x_j \in B_j$ and $B_i \subseteq U_i$. Then
$x \in Z[B_1 \times \cdots \times B_m, \emptyset]$, so by convergence there exists
$N \in \mathbb{N}$ such that $x^n \in Z[B_1 \times \cdots \times B_m,\emptyset]$ for all
$n \ge N$. Thus $\ell(x^n) \ge m$ for all $n \ge N$, hence $\ell(x^n)\to\infty$, and
$x_i^n \in B_i \subseteq U_i$ for all large $n$, proving that
$(x_i^n)_{n : \ell(x^n) \ge i} \to x_i$ in $X$.

Conversely, suppose $\ell(x^n)\to \infty$ and
$(x_i^n)_{n : \ell(x^n) \ge i} \to x_i$ in $X$ for all $i \in \mathbb{N}$. Let $U$ be an open
neighborhood of $x$ in $W_0$. By Proposition~\ref{lemabaseviz}, there is a basic open set
$Z[B_1 \times \cdots \times B_k, K]$ such that
$x \in Z[B_1 \times \cdots \times B_k, K] \subseteq U$. Since $\ell(x^n)\to \infty$, we may
assume $\ell(x^n)\ge k$ for all $n$ sufficiently large. Also, by coordinatewise convergence,
$x_i^n \in B_i$ for all $i\le k$ and large $n$, and $x_{k+1}^n \notin K$ for large $n$,
since $x_{k+1} \in X \setminus K$. Hence $x^n \in Z[B_1 \times \cdots \times B_k, K]$ for all
sufficiently large $n$, and therefore $x^n \to x$, completing the proof.
\end{proof}

Before introducing the metric on $W_0$, we record the following standard fact, which allows
one to transport a metric along a homeomorphism.

\begin{remark}\label{rem metrica e topologia} Let $(T,\tau)$ be a topological space, let $(Y,d_Y)$ be a metric space, and let
$h:T \to Y$ be a homeomorphism. Define
$
d_T(x,y) := d_Y\big(h(x),h(y)\big)$, 
 for all $ x,y \in T.$
Then $d_T$ is a metric on $T$, and the topology induced by $d_T$ coincides with $\tau$.
In other words, $d_T$ is the pullback of $d_Y$ via $h$, and $h$ becomes an isometry from
$(T,d_T)$ onto $(Y,d_Y)$.
\end{remark}

We now define a metric on $W_0$. Recall that
\[
\mathcal{P} = \bigcup_{n=1}^\infty \mathcal{B}^n
\]
is countable, so we may fix an enumeration
\(\mathcal{P} = \{p_1, p_2, \ldots\}\).
We consider on $\{0,1\}^{\mathcal{P}}$ a metric that induces the product topology, namely
\begin{equation}\label{eq:d_fin}
d_{\mathrm{fin}}(\mu,\nu)
=
\begin{cases}
\frac{1}{2^i}
&\text{if } i \in \mathbb{N} \text{ is the smallest index such that } \mu_i \neq \nu_i,\\[4pt]
0
&\text{if } \mu_i = \nu_i \text{ for all } i \in \mathbb{N},
\end{cases}
\end{equation}
and transport this metric via the homeomorphism $\alpha$ to $W_0$. We make this precise below.

\begin{corollary}\label{w0 eh metrico}
Let $\alpha : W_0 \to \alpha(W_0)$ be the homeomorphism of
Proposition~\ref{propalpha}. Define
\[
d(x,y) := d_{\mathrm{fin}}\big(\alpha(x), \alpha(y)\big)
\quad\text{for all } x,y \in W_0.
\]
Then $d$ is a metric on $W_0$, and the topology it induces coincides with the generalized
cylinder set topology on $W_0$.
\end{corollary}

\begin{proof}
This follows directly from Proposition~\ref{same topologies} and
Remark~\ref{rem metrica e topologia}.
\end{proof}

\begin{remark}\label{metric}
From the definition of $d$, for $x,y \in W_0$ we have
\[
d(x,y)=
\begin{cases}
\frac{1}{2^i}
&\text{if $i \in \mathbb{N}$ is the smallest index such that exactly one of }
x,y \text{ belongs to } Z[p_i,\emptyset],\\[4pt]
0 &\text{if } x=y.
\end{cases}
\]
In particular, if $x \in Z[p_j,\emptyset]$ and $y \notin Z[p_j,\emptyset]$ for some
$j \in \mathbb{N}$, then $d(x,y) \ge \frac{1}{2^j}$.
\end{remark}

\begin{example}\label{OTW}
Let $X=\alpha$, where $\alpha$ is a countably infinite set equipped with the discrete topology.
We show that, in this case, our space $W_0$ coincides with the full one-sided shift over
$\alpha$ introduced by Ott--Tomforde--Willis in \cite{OTW14}.

Since $X=\alpha$ is discrete, its compact open subsets are precisely the finite subsets.
Thus we may take
\[
\mathcal B=\big\{\{a\}: a\in \alpha\big\},
\]
which is a countable basis of compact open sets closed under finite intersections. The
one-point compactification $\alpha_\infty=\alpha\cup\{\infty\}$ is the space used in the OTW
construction: every finite word in $\alpha$ is identified with the corresponding sequence in
$\alpha_\infty^{\mathbb N}$ that is eventually equal to $\infty$, while infinite words are
identified with elements of $\alpha^{\mathbb N}$. Hence the underlying set of $W_0$
coincides with the space denoted by $\Sigma_\alpha$ in \cite[\S2]{OTW14}.

We next compare the topologies. Since each element of $\mathcal B$ is a singleton, the sets in $
\mathcal P=\bigcup_{n\ge1}\mathcal B^n
$
identify naturally with nonempty finite words over $\alpha$. Moreover, for
$a_1,\dots,a_k\in\alpha$ and a finite set $K\subseteq \alpha$, we have
\[
Z[\{a_1\}\times\cdots\times\{a_k\},K]
=
\{x\in W_0 : x_i=a_i \text{ for } 1\le i\le k,\ x_{k+1}\notin K\},
\]
which is exactly the OTW generalized cylinder $Z(a_1,\dots,a_k;K)$
\cite[Def.~2.8]{OTW14}. Hence the generalized cylinder topology on $W_0$
coincides with the OTW topology on $\Sigma_\alpha$.

Finally, the metric on $W_0$ is defined by embedding $W_0$ into
$\{0,1\}^{\mathcal P}$ and pulling back the standard product metric via an enumeration of
$\mathcal P$. In OTW, the metric on $\Sigma_\alpha$ is defined in the same way, using an
enumeration of the set of finite words \cite[\S2.3]{OTW14}. Under the identification
$\mathcal P\cong \Sigma_\alpha^{\mathrm{fin}}\setminus\{\vec 0\}$ described above, and for
corresponding choices of enumeration, our metric coincides with the OTW metric. 
Thus, in the discrete case, $W_0$ is exactly the OTW
full one-sided shift, both topologically and metrically.
\end{example}

We finish this section by proving that the metric defined above is an ultrametric.

Recall that a metric $d$ on a set $M$ is called an ultrametric if it satisfies a stronger version of the triangle inequality, namely $
d(x,z) \le \max\{d(x,y),\, d(y,z)\}
\quad \text{for all } x,y,z \in M.$
In this case, the pair $(M,d)$ is called an ultrametric space. For $x \in M$ and
$r>0$, the (open) ball centered at $x$ of radius $r$ is denoted by $B(x,r) := \{\, y \in M : d(x,y) < r \,\}.$
It is well known that ultrametric spaces satisfy the following properties.

\begin{lemma}\label{ultramentric prop}
Let $(M,d)$ be an ultrametric space, and let $x,y \in M$ and $r>0$. Then:
\begin{enumerate}
\item Either $B(x,r)\cap B(y,r)=\emptyset$ or $B(x,r)=B(y,r)$.
\item If $z \in B(x,r)$, then $B(z,r)=B(x,r)$.
\item If $a \notin B(x,r)$ and $b \in B(x,r)$, then $d(a,b)\ge r$.
\end{enumerate}
\end{lemma}

\begin{proposition}\label{W_0 is ultrametric}
Let $X$ and $W_0$ be as in Section~\ref{W0 section}. Then the metric $d$ defined in
Remark~\ref{metric} is an ultrametric on $W_0$. Consequently, $W_0$ is an ultrametric space.
\end{proposition}

\begin{proof}

Since $d$ is already known to be a metric, it remains to verify the ultrametric inequality:
\[
d(x,y) \le \max\{d(x,z),\,d(z,y)\}
\quad\text{for all } x,y,z \in W_0.
\]

Let $x,y \in W_0$ with $x \ne y$, and let $i \in \mathbb{N}$ be the smallest index such that
$d(x,y) = \frac{1}{2^i}$. Then precisely one of $x,y$ lies in $Z[p_i,\emptyset]$; without
loss of generality, suppose $x \in Z[p_i,\emptyset]$ and $y \notin Z[p_i,\emptyset]$.

Now let $z \in W_0$. Then either:
\begin{itemize}
\item $z \in Z[p_i,\emptyset]$, in which case $d(z,y) \ge \frac{1}{2^i}$; or
\item $z \notin Z[p_i,\emptyset]$, in which case $d(z,x) \ge \frac{1}{2^i}$.
\end{itemize}
In both cases we have
\[
d(x,y) = \frac{1}{2^i} \le \max\{d(x,z),d(z,y)\},
\]
establishing the ultrametric inequality. 
\end{proof}

\begin{corollary}
    It follows from the proposition above and Example~\ref{OTW} that the OTW-subshifts of \cite{OTW14} are ultrametric spaces.
\end{corollary}

\section{The inverse limit and its shifts}
\label{tornado}

In this section we associate an inverse-limit-type space to a Deaconu--Renault system and show that it carries a natural shift dynamics.

\begin{definition}
A Deaconu--Renault system is a pair $(X,f)$ consisting of a locally compact
Hausdorff space $X$ and a map
$
f:\Dom(f)\longrightarrow \Im(f),
$
where $\Dom(f)$ (the domain of $f$) and $\Im(f)$ (the image of $f$) are open subsets of $X$, such that $f$ is a local homeomorphism.
\end{definition}

Throughout this section, let $(X,f)$ be a Deaconu--Renault system and put
\[
D:=\Dom(f).
\]
We assume that $D$ is infinite and satisfies the hypotheses of Section~\ref{W0 section},
namely: $D$ is locally compact Hausdorff and admits a countable basis $\mathcal B$ of
compact open sets closed under finite intersections. Hence the space $D_0$ obtained by applying the construction of 
Section~\ref{W0 section} (and there called $W_0$) to $D$ is available. We also assume throughout that
\[
D\subseteq \Im(f).
\]
This assumption is automatic for surjective Deaconu--Renault systems, which form a large and natural class of examples; see \cite{AraClaramunt2024, ArmstrongBrixCarlsenEilers2023}.

We now define the inverse limit of $(X,f)$ inside $D_0$.

\begin{definition}\label{def:inverse-limit}
A word $x=x_1x_2\cdots\in D_0$ of positive length is called a backward $f$--path if $
f(x_{i+1})=x_i$ for all $1\le i<\ell(x)$.

We define
\[
D_\infty:=\{x\in D_0:\ \ell(x)=\infty \text{ and }x\text{ is a backward }f\text{--path}\},
\]
and
\[
D_{\mathrm{fin}}:=\{x\in D_0:\ 1\le \ell(x)<\infty \text{ and }x\text{ is a backward }f\text{--path}\}\cup\{\vec 0\}.
\]

A word $x\in D_{\mathrm{fin}}$ is called limit-type if there exists a sequence
$(x^{(n)})_{n\in\mathbb N}\subseteq D_\infty$ such that
$
x^{(n)}\longrightarrow x $ in $D_0.$
We denote the set of such words by
\[
D_{\mathrm{lim}}
:=\{x\in D_{\mathrm{fin}}:\exists (x^{(n)})_{n\in\mathbb N}\subseteq D_\infty
\text{ with }x^{(n)}\to x\}.
\]

Finally, we define the inverse limit of $(X,f)$ by
\[
\widetilde X:=D_\infty\cup D_{\mathrm{lim}}.
\]
\end{definition}

\begin{remark}\label{rem:finite-length-under-D-in-Im}
The space $\widetilde X$ usually contains words
of finite positive length.
For example, let $X=\mathbb N$ with the discrete topology, let $D=\mathbb N$, and define
\[
f(n)=1 \ \text{ for $n$ odd},\qquad
f(m)=\frac m2 \ \text{ for $m$ even}.
\]
Then $\Dom(f)=\Im(f)=\mathbb N$, so our standing assumption holds.
For each $n\in\mathbb N$, let $m_n:=2n+1$ and consider
\[
x^{(n)}=(1,m_n,2m_n,4m_n,8m_n,\dots)\in D_\infty.
\]
By Corollary~\ref{convergencia em W0}, $x^{(n)}\to (1)$ in $D_0$. Hence $(1)\in D_{\mathrm{lim}}$.
A similar argument shows that $
\vec 0,\ (1),\ (1,1),\ (1,1,1),\dots$
all belong to $D_{\mathrm{lim}}$.
\end{remark}

\begin{proposition}\label{prop:closure-Dinfty}
$\overline{D_\infty}^{\,D_0} = D_\infty \cup D_{\mathrm{lim}}.$
\end{proposition}

\begin{proof}
We only need to prove that $\overline{D_\infty}^{\,D_0} \subseteq D_\infty \cup D_{\mathrm{lim}}$,
since $D_\infty \cup D_{\mathrm{lim}} \subseteq \overline{D_\infty}^{\,D_0}$ by definition.

Let $x \in \overline{D_\infty}^{\,D_0}$. Then there exists a sequence
$(x^{(n)})_{n\in\mathbb{N}} \subseteq D_\infty$ such that $x^{(n)} \to x$ in $D_0$.
Write $x = x_1x_2\cdots$ and $x^{(n)} = x^{(n)}_1x^{(n)}_2\cdots$.

We first show that $x$ is a backward $f$-path. Let $i < \ell(x)$.
By Corollary~\ref{convergencia em W0}, we have
$x^{(n)}_i \to x_i$ and $x^{(n)}_{i+1} \to x_{i+1}$ in $D$.
Since each $x^{(n)} \in D_\infty$, we have
$f(x^{(n)}_{i+1}) = x^{(n)}_i \quad \text{for all } n.$
Passing to the limit and using the continuity of $f$ on $D$, we obtain $
f(x_{i+1}) = x_i.$ 
Thus $x$ is a backward $f$-path.

If $\ell(x)=\infty$, then $x \in D_\infty$ and we are done. Suppose now that $\ell(x)=k<\infty$.
Since $(x^{(n)}) \subseteq D_\infty$ and $x^{(n)} \to x$, it follows that
$x \in \overline{D_\infty}^{\,D_0}$ and $x$ has finite length.
By definition of $D_{\mathrm{lim}}$, this means $x \in D_{\mathrm{lim}}$.

Therefore $x \in D_\infty \cup D_{\mathrm{lim}}$ in all cases, proving that
$\overline{D_\infty}^{\,D_0} \subseteq D_\infty \cup D_{\mathrm{lim}}$.
\end{proof}

\begin{corollary}\label{cor:Xtilde-compact}
The set $\widetilde X$ is closed in $D_0$, and hence compact.
\end{corollary}

\begin{proof}
By Proposition~\ref{prop:closure-Dinfty},
$\widetilde X=\overline{D_\infty}^{\,D_0}$,
so $\widetilde X$ is closed in $D_0$. Since $D_0$ is compact, $\widetilde X$ is compact.
\end{proof}

\begin{remark}\label{remarkXtildeultramet}
Since $D_0$ is ultrametric by Proposition~\ref{W_0 is ultrametric}, the subspace
$\widetilde X$ is also ultrametric.
\end{remark}

The next lemma shows that $\widetilde X$ is stable under taking tails and under prefixing by the map $f$.

\begin{lemma}\label{lem:prefix-tail}
Let $(X,f)$ be a Deaconu--Renault system such that $\Dom(f)\subseteq \Im(f)$, and let $\widetilde{X}$ be as in Definition \ref{def:inverse-limit}. The following hold.
\begin{enumerate}
\item[(i)] If $x\in \widetilde X$ and $\ell(x)\ge 1$, then $
f(x_1)x:=f(x_1)x_1x_2\cdots \in \widetilde X.$ 
More precisely, if $x\in D_\infty$ then $f(x_1)x\in D_\infty$, and if $x\in D_{\mathrm{lim}}$
then $f(x_1)x\in D_{\mathrm{lim}}$.

\item[(ii)] If $x\in \widetilde X$ and $\ell(x)\ge 1$, then its tail $
x_2x_3\cdots \in \widetilde X.$
More precisely, if $x=x_1x_2\ldots\in D_\infty$ then $x_2x_3\cdots\in D_\infty$. If
$x=x_1x_2\ldots\in D_{\mathrm{lim}}$ has finite length at least $2$, then $x_2x_3\cdots\in D_{\mathrm{lim}}$, and if $x$ has length $1$ then $\vec{0}\in D_{\mathrm{lim}}$.
\end{enumerate}
\end{lemma}

\begin{proof}
(i) If $x\in D_\infty$, then
$
f(x_1)x=f(x_1)x_1x_2x_3\cdots
$
is again an infinite backward $f$--path, hence belongs to $D_\infty$.

Suppose now that $x\in D_{\mathrm{lim}}$ and $\ell(x)\ge 1$. Choose
$(x^{(n)})_{n\in\mathbb N}\subseteq D_\infty$ with $x^{(n)}\to x$ in $D_0$, and write
\[
x^{(n)}=x^{(n)}_1x^{(n)}_2\cdots,\qquad x=x_1x_2\cdots.
\]
Define $y^{(n)}:=f(x^{(n)}_1)x^{(n)}_1x^{(n)}_2\cdots.$
Then $y^{(n)}\in D_\infty$ for every $n$.
By Corollary~\ref{convergencia em W0}, $x^{(n)}_1\to x_1$ in $D$, and hence $
f(x^{(n)}_1)\to f(x_1)$
by continuity of $f$. Applying again Corollary~\ref{convergencia em W0}, we obtain
$
y^{(n)}\to f(x_1)x$ in $D_0$.
Therefore $f(x_1)x\in D_{\mathrm{lim}}$.

(ii) If $x\in D_\infty$, then its tail is clearly still an infinite backward $f$--path, hence
belongs to $D_\infty$.

Suppose now that $x\in D_{\mathrm{lim}}$ and $\ell(x)\ge 2$. Choose
$(x^{(n)})_{n\in\mathbb N}\subseteq D_\infty$ with $x^{(n)}\to x$, and define
$
y^{(n)}:=x^{(n)}_2x^{(n)}_3\cdots\in D_\infty.$
By Corollary~\ref{convergencia em W0}, the sequence $(y^{(n)})_{n\in\N}$ converges in $D_0$ to the tail
$x_2x_3\cdots$. Hence $x_2x_3\cdots\in D_{\mathrm{lim}}$.

 Finally, suppose that $\ell(x)=1$, and let $(x^n)_{n\in \N}\subseteq D_\infty$ be such that $x^n\rightarrow x$. Write $x^n=x_1^nx_2^nx_3^n$ and $x=x_1$. Define $z^n=x_2^nx_3^n...$, notice that $(z^n)_{n\in\N}\subseteq D_\infty$, and from Corollary \ref{convergencia em W0} we get that $z^n\rightarrow \vec{0}$, so that $\vec{0}\in \widetilde{X}$.

\end{proof}

\begin{remark} Notice that it is possible that $\vec{0}\in \widetilde{X}$ even if $\widetilde{X}$ does not contain any element of finite strict positive length. As an example, let $X=\N$ and let $f:X\rightarrow X$ be defined by $f(x)=x$ for each $x\in X$. Notice that $D_\infty=\{(nnn...): n\in \N\}$. Let $x^n=nnn...$ and it is not hard to see that $D_{\lim}=\{\vec{0}\}$. Hence, $\widetilde{X}=\{(nnn...): n\in \N\}\cup \{\vec{0}\}.$

\end{remark}

We now define the shift maps on $\widetilde X$.

\begin{definition}\label{def:sigma-alpha}
Let $(X,f)$ be a Deaconu--Renault system such that $D:=\Dom(f)\subseteq \Im(f)$, and let
$\widetilde X$ be as in Definition~\ref{def:inverse-limit}.

We define $
\Dom(\sigma):=\{x\in \widetilde X:\ell(x)\ge 1\},$
and
$
\sigma:\Dom(\sigma)\longrightarrow \widetilde X$
by
\[
\sigma(x_1x_2x_3\cdots):=
\begin{cases}
\vec 0, & \text{if }\ell(x)=1,\\[4pt]
x_2x_3\cdots x_k, & \text{if }\ell(x)=k\ge 2,\\[4pt]
x_2x_3x_4\cdots, & \text{if }\ell(x)=\infty.
\end{cases}
\]
By Lemma~\ref{lem:prefix-tail}(ii) and the assumption $D\subseteq \Im(f)$, this map is well defined.

We set
\[
E:=\sigma(\Dom(\sigma))\subseteq \widetilde X,
\]
endowed with the subspace topology.

We also define
\[
\Dom(\widehat\sigma):=\{x\in \widetilde X:\ell(x)\ge 2\}\subseteq \Dom(\sigma),
\]
and
\[
\widehat\sigma:=\sigma|_{\Dom(\widehat\sigma)}:\Dom(\widehat\sigma)\to E\setminus\{\vec 0\}.
\]

For every $y=y_1y_2\cdots\in E\setminus\{\vec 0\}$ we define
\[
\alpha_f(y):=f(y_1)y_1y_2\cdots.
\]
By Lemma~\ref{lem:prefix-tail}(i), this map is well defined and takes values in $\widetilde X$.
\end{definition}

\begin{remark}\label{rem:sigma-hatsigma}
The choice $\Dom(\sigma)=\{x\in \widetilde X:\ell(x)\ge 1\}$
is important because it preserves the point $\vec 0$ in the image, which will be needed later.
However, $\sigma$ is not the appropriate inverse of $\alpha_f$ at words of length $1$.
For this reason we introduce the truncated shift
$
\widehat\sigma=\sigma|_{\{x\in \widetilde X:\ell(x)\ge 2\}},$ 
and it is $\widehat\sigma$, rather than $\sigma$, that is inverse to $\alpha_f$.
\end{remark}

\begin{proposition}\label{prop:shift-DR-correct}
Let $(X,f)$ be a Deaconu--Renault system such that $D:=\Dom(f)\subseteq \Im(f)$.
With the notation above, the following hold:
\begin{enumerate}[(i)]
\item
\[
\Dom(\sigma)=
\begin{cases}
\widetilde X\setminus\{\vec 0\}, & \text{if }\vec 0\in \widetilde X,\\[4pt]
\widetilde X, & \text{if }\vec 0\notin \widetilde X.
\end{cases}
\]
In particular, $\Dom(\sigma)$ is an open locally compact subset of $\widetilde X$.

\item
Every element of $\widetilde X$ of positive length belongs to $E$. Hence
\[
E=\widetilde X
\quad\text{or}\quad
E=\widetilde X\setminus\{\vec 0\}.
\]
Moreover, $
\vec 0\in E$ iff $
\widetilde X $  contains a word of length $1$.

In all cases, $E$ is locally compact (compact if $\vec 0\in E$)  and $E\setminus\{\vec 0\}$ is open in $E$.

\item
The maps
$
\widehat\sigma:\Dom(\widehat\sigma)\longrightarrow E\setminus\{\vec 0\}$ and $
\alpha_f:E\setminus\{\vec 0\}\longrightarrow \Dom(\widehat\sigma)
$
are mutually inverse homeomorphisms.

\item
$(E,\alpha_f)$ is a Deaconu--Renault system.
More precisely, $\alpha_f$ is a homeomorphism from the open subset
$\Dom(\alpha_f)=E\setminus\{\vec 0\}\subseteq E$ onto the open subset
\[
\Im(\alpha_f)=\Dom(\widehat\sigma)=\{x\in \widetilde X:\ell(x)\ge 2\}\subseteq E.
\]
\end{enumerate}
\end{proposition}

\begin{proof}
\textit{(i)} By definition,
\[
\Dom(\sigma)=\{x\in \widetilde X:\ell(x)\ge 1\}.
\]
Thus, if $\vec 0\in \widetilde X$, then $\Dom(\sigma)=\widetilde X\setminus\{\vec 0\}$; otherwise
$\Dom(\sigma)=\widetilde X$.

Since $\widetilde X$ is compact Hausdorff, the singleton $\{\vec 0\}$ is closed in $\widetilde X$.
Hence $\Dom(\sigma)$ is open in $\widetilde X$, and therefore locally compact.

\smallskip
\textit{(ii)} Let $x\in \widetilde X$ with $\ell(x)\ge 1$, say $x=x_1x_2\cdots$.
By Lemma~\ref{lem:prefix-tail}(i),
$
z:=f(x_1)x_1x_2\cdots\in \widetilde X.$ 
Moreover, $\ell(z)\ge 2$, so $z\in \Dom(\sigma)$ and
$
\sigma(z)=x.$
Thus every element of $\widetilde X$ of positive length belongs to $E$.

Therefore $E=\widetilde X$ or $
E=\widetilde X\setminus\{\vec 0\}.$ Also, by definition of $\sigma$,
\[
\vec 0\in E
\ \Longleftrightarrow\
\text{there exists }x\in \Dom(\sigma)\text{ with }\ell(x)=1
\ \Longleftrightarrow\ 
\widetilde X \text{ contains a word of length }1.
\]
If $\vec 0\in E$, then the previous paragraph gives $E=\widetilde X$, so $E$ is closed in $\widetilde X$.
In any case, $E$ is either $\widetilde X$ or $\widetilde X\setminus\{\vec 0\}$, hence locally compact.
Moreover, $E\setminus\{\vec 0\}$ is open in $E$.

\smallskip
\textit{(iii)} Let $y\in E\setminus\{\vec 0\}$. Then $\ell(y)\ge 1$, so $\alpha_f(y)$ is defined, and
$
\ell(\alpha_f(y))\ge 2.
$
Hence $\alpha_f(y)\in \Dom(\widehat\sigma)$ and
$
\widehat\sigma(\alpha_f(y))
=\sigma(f(y_1)y_1y_2\cdots)
=y.
$

Now let $x=x_1x_2x_3\cdots\in \Dom(\widehat\sigma)$, so $\ell(x)\ge 2$.
Then
\[
\alpha_f(\widehat\sigma(x))
=\alpha_f(x_2x_3\cdots)
=f(x_2)x_2x_3\cdots.
\]
Since $x$ is a backward $f$--path, we have $f(x_2)=x_1$, so
$
\alpha_f(\widehat\sigma(x))=x_1x_2x_3\cdots=x.
$

It remains to prove continuity.

Let $(x^{(n)})_{n\in \N}\subseteq \Dom(\widehat\sigma)$ with $x^{(n)}\to x\in \Dom(\widehat\sigma)$ in $\widetilde X$.
Write
\[
x^{(n)}=x^{(n)}_1x^{(n)}_2\cdots,
\qquad
x=x_1x_2\cdots.
\]
By Corollary~\ref{convergencia em W0}, for each $i\leq \ell(x)$ we have
$
x^{(n)}_{i+1}\to x_{i+1}.$
Hence again by Corollary~\ref{convergencia em W0},
$\widehat\sigma(x^{(n)})\to \widehat\sigma(x)$ in $D_0,
$
and therefore in $\widetilde X$. Thus, $\widehat\sigma$ is continuous.

Now, let $(y^{(n)})_{n\in \N}\subseteq E\setminus\{\vec 0\}$ with $y^{(n)}\to y\in E\setminus\{\vec 0\}$.
Write
\[
y^{(n)}=y^{(n)}_1y^{(n)}_2\cdots,
\qquad
y=y_1y_2\cdots.
\]
By Corollary~\ref{convergencia em W0},
$
y^{(n)}_1\to y_1
$ in $D,
$
and hence, by continuity of $f$,
$
f(y^{(n)}_1)\to f(y_1).
$
Applying Corollary~\ref{convergencia em W0} once more, we obtain
$
\alpha_f(y^{(n)})\to \alpha_f(y)
$ in $D_0,
$
hence in $\widetilde X$. Therefore $\alpha_f$ is continuous.

Thus $\widehat\sigma$ and $\alpha_f$ are mutually inverse homeomorphisms.

\smallskip
\textit{(iv)} By \textit{(ii)}, 
$
\Dom(\alpha_f)=E\setminus\{\vec 0\}
$
is open in $E$.
By \textit{(iii)}, its image is
\[
\Im(\alpha_f)=\Dom(\widehat\sigma)=\{x\in \widetilde X:\ell(x)\ge 2\}.
\]
Since the set of words of length at most $1$ is closed in $\widetilde X$, it follows that
$\Dom(\widehat\sigma)$ is open in $\widetilde X$, hence also open in $E$.
Finally, by \textit{(iii)}, $\alpha_f$ is a homeomorphism from $\Dom(\alpha_f)$ onto $\Im(\alpha_f)$.
Therefore $(E,\alpha_f)$ is a Deaconu--Renault system.
\end{proof}

\section{Shadowing via defining sequences}

In this section, we show that in ultrametric spaces with a tame defining sequence, one can formulate shadowing purely in terms of the partitions in the defining sequence, without restricting attention to compact subsets. 
The defining sequence provides a global, countable structure that replaces compactness in the classical metric setting (see \cite{DGS, GoodMeddaugh2020}).

Recall that a partition $\mathcal U$ of a space $X$ is a countable family of pairwise disjoint, nonempty, clopen sets whose union is $X$. For $x \in X$, we denote by $\mathcal U[x]$ the carrier of $x$, that is, the unique element of $\mathcal U$ which contains $x$. Below we recall the definition of tame defining sequences, as in \cite{DGS}.

\begin{definition}
    \label{nintendoDS}
A defining sequence of a space $X$ is a sequence $\mathcal A = \{\m{U}{n}\}_{n \in \N}$ of partitions of $X$ such that:
\begin{enumerate}
    \item $\m{U}{n+1}$ refines $\m{U}{n}$, i.e., each element of $\m{U}{n+1}$ is contained in a (necessarily unique) element of $\m{U}{n}$.
    \item The family $\{U : U \in \m{U}{n} \text{ for some } n \in \N\}$ is a basis for the topology of $X$.
\end{enumerate}
The defining sequence $\mathcal A$ is called complete if, whenever $U_n \in \m{U}{n}$ with $U_{n+1} \subseteq U_n$ for all $n \in \N$, one has $\bigcap\limits_{n\in \N} U_n \neq \emptyset$.
\end{definition}

\begin{definition}\label{tameDS}
Let $(X,d)$ be a metric space and let $\{\m{U}{n}\}_{n \in \N}$ be a defining sequence of $X$. For each $n \in \N$, set
\[
S_n = \sup\{\mathrm{diam}(O) : O \in \m{U}{n}\},
\]
where $\mathrm{diam}(O)$ denotes the diameter of $O$ with respect to $d$.
We say that $\{\m{U}{n}\}_{n \in \N}$ is a tame defining sequence of $X$ (with respect to $d$) if $S_n \to 0$ as $n \to \infty$ and, for each $n \in \N$, there exists $\rho_n > 0$ such that whenever $O_1, O_2$ are distinct elements of $\m{U}{n}$ and $x_i \in O_i$, then
\[
d(x_1,x_2) \ge \rho_n.
\]
In this case we say that $\m{U}{n}$ is $\rho_n$-separated. 
\end{definition}

\begin{example}
    Let $X=\{m+\frac{1}{n}: m,n\in \N \text{ with } n\geq 2\}\cup \N \subseteq \R$ with the usual metric. Fix an $n\in \N$, and define, for each $m\in \N$, $V_m=\{m+\frac{1}{k}:k\geq n\}\cup \{m\}$, and let $$\mathcal{U}_n=\{V_m:m\in\N\}\cup \{\{m+\frac{1}{p}\}: m\in \N \text{ and }1\leq p<n\}.$$
Notice that each $\mathcal{U}_n$ is a clopen partition of $X$, that $\sup\{ \operatorname{diam}(V): V\in \mathcal{U}_n \}=\frac{1}{n}$ and that $d(U,V)\geq\frac{1}{n-1}-\frac{1}{n}=\frac{1}{(n-1)n}$ for each $U,V\in \mathcal{U}_n$. Moreover, $\mathcal{U}_{n+1}$ refines $\mathcal{U}_n$ and so $\{\mathcal{U}_n\}_{n\in\N}$ is a complete tame defining sequence.

\end{example}

Next, we recall the definition of shadowing in Deaconu--Renault systems, as defined in \cite{GU}.

\begin{definition}\label{definicial} Let $(X,f)$ be a Deaconu--Renault system, with $X$ a metric space.  Given $\delta>0$, a (finite) \textit{$\delta$}-pseudo-orbit in $(X,f)$ is a (finite) sequence $(x_n)$ such that $d(f(x_n),x_{n+1})< \delta$ for all $n$, where it is implicit that $x_n\in \Dom(f)$ when $f(x_n)$ is written. 

We say that a point $x\in X$ \textit{$\varepsilon-$shadows} a (finite) sequence $(y_n) $ in $X$ if $d(f^n(x), y_n)<\varepsilon$ for all $n$. Notice that this implies that $x\in \Dom(f^n)$ for all $n$.
\end{definition}

In \cite{GoodMeddaugh2020}, shadowing on compact spaces is shown to be a topological property, independent of the particular compatible metric. Below we extend this result to spaces with tame defining sequences.

\begin{definition}\label{def:U-shadowing}
Let $(X,f)$ be a Deaconu--Renault system. Let $\mathcal{U}$ be a partition of $\Dom(f)$
into clopen sets. 
\begin{itemize}
\item A sequence $(x_n)_{n\in\N} \subseteq \Dom(f)$ is called a 
$\mathcal{U}$-pseudo-orbit if for every $n$ there exists 
$U_n \in \mathcal{U}$ such that $\{f(x_n), x_{n+1}\} \subseteq U_n$.

\item A point $z \in \Dom(f)$ is said to $\mathcal{U}$-shadow the sequence $(x_n)_{n\in \N}$ if
for every $n$ there exists $W_n \in \mathcal{U}$ such that 
$\{f^n(z), x_n\} \subseteq W_n$.
\end{itemize}
\end{definition}

The following generalizes \cite[Lemma~6]{GoodMeddaugh2020}.

\begin{proposition}[Shadowing characterization]\label{prop:shadowing-char}
Let $(X,f)$ be a Deaconu--Renault system such that $\Dom(f)$ admits a 
tame defining sequence $\{\mathcal{U}_n\}_{n\in\mathbb{N}}$.  
Then the following statements are equivalent:

\begin{enumerate}
\item[(M)] Metric shadowing:
For every $\varepsilon>0$ there exists $\delta>0$ such that every 
$\delta$-pseudo-orbit in $\Dom(f)$ is $\varepsilon$-shadowed by some $z\in\Dom(f)$.

\item[(T)] Shadowing via the defining sequence:
For every $n\in\mathbb{N}$ there exists $m\ge n$ such that every 
$\mathcal{U}_m$-pseudo-orbit is $\mathcal{U}_n$-shadowed by some $z\in\Dom(f)$.
\end{enumerate}
\end{proposition}

\begin{proof}
(M $\Rightarrow$ T): Fix $n$ and set $\varepsilon := \rho_n/2$.  
By (M) there exists $\delta>0$ such that every $\delta$-pseudo-orbit 
is $\varepsilon$-shadowed.  
Choose $m>n$ with $S_m<\delta$.  
Then every $\mathcal{U}_m$-pseudo-orbit is a $\delta$-pseudo-orbit, 
hence $\varepsilon$-shadowed by some $z$.  
Since distinct elements of $\mathcal{U}_n$ are $\rho_n$-separated,  
$d(f^i(z),x_i)<\rho_n/2$ implies $f^i(z)$ and $x_i$ lie in the same atom of 
$\mathcal{U}_n$.  
Thus the pseudo-orbit is $\mathcal{U}_n$-shadowed.

(T $\Rightarrow$ M): Let $\varepsilon>0$.  
Choose $n$ such that $S_n\le\varepsilon$.  
By (T) there exists $m\ge n$ such that every $\mathcal{U}_m$-pseudo-orbit is 
$\mathcal{U}_n$-shadowed.  
Let $\delta:=\rho_m$.  
If $(x_i)$ is a $\delta$-pseudo-orbit, then 
$d(f(x_i),x_{i+1})<\rho_m$
forces $f(x_i),x_{i+1}$ to lie in the same atom of $\mathcal{U}_m$,  
so $(x_i)$ is a $\mathcal{U}_m$-pseudo-orbit.  
By assumption it is $\mathcal{U}_n$-shadowed by some $z$, hence
$d(f^i(z),x_i)\le S_n\le\varepsilon$ for all $i$, proving (M).
\end{proof}

We finish this section recording how to build standard defining sequences in an ultrametric space.

\begin{proposition}\label{prop:ball-tame-defining}
Let $(X,d)$ be an ultrametric space, let $D\subseteq X$ be a nonempty open subset, and let
$(a_n)_{n\ge 1}\subseteq (0,\infty)$ be a decreasing sequence such that $a_n\to 0$.
For each $n\ge 1$, define
\[
\mathcal U_n:=\{B_D(x,a_n):x\in D\};
\]
After removing repetitions, each $\mathcal U_n$ is a countable clopen partition of $D$.
Moreover, if we set $\mathcal U_0:=\{D\},$
then $\mathcal A:=(\mathcal U_n)_{n\ge 0}$ is a tame defining sequence for $D$.

If, in addition, $D$ is complete, then $\mathcal A$ is a complete tame defining sequence. In particular, this applies to the spaces $D_0$ and $\widetilde X$ of Section~\ref{tornado}, since both are compact ultrametric spaces.
\end{proposition}

\begin{proof}
For the special choice $a_n=1/n$, this is exactly \cite[Proposition~3.1.11]{DGS}. The reader can verify that the same proof works verbatim for an arbitrary decreasing sequence $a_n\downarrow 0$.

If $D$ is complete, then the same completeness argument as in
\cite[Propositions~3.1.6 and 3.1.11]{DGS} shows that $\mathcal A$ is complete. The final assertion follows because $D_0$ and $\widetilde X$ are compact ultrametric spaces.
\end{proof}

\section{Shadowing via inverse limits for ultrametric Deaconu--Renault systems}\label{shadowthis}

The goal of this section is to relate shadowing for a Deaconu--Renault system to shadowing for its associated inverse-limit system. More precisely, under suitable hypotheses, we show that shadowing for the compactified base system is equivalent to shadowing for the inverse-limit system. The main additional assumption is a uniform expansiveness condition, formulated in terms of contractions of the local inverse branches.

We begin by introducing the first $L-$coordinate projection from the inverse limit $\widetilde{X}$ of Section~\ref{tornado} to the one-point compactification $D^\infty$. We then define the separation property mentioned above and establish the continuity of the projection. These ingredients will be used in the proof of the main theorem of the section.

Let $(X,f)$ be a DR-system with domain $D$ as in Section~\ref{tornado}. Recall the space $D_0$ and the inverse limit definition of $\widetilde{X}$.
Let $D^\infty$ denote the one-point compactification of $D$, that is,
$D^\infty = D \cup \{\infty\}$.
For each $x=(x_1,x_2,...)\in \widetilde{X}$, and for each $n \in \N$ define  
\[
\pi_n(x) :=
\begin{cases}
x_n, & \text{if } \ell(x)\geq n\\
\infty, & \text{if } \ell(x)<n.
\end{cases}
\]

Fix an $L\in \N$, and define $\pi_L^\infty:\widetilde{X}\rightarrow (D^\infty)^L$  by $\pi_L^\infty(x)=(\pi_1(x), \pi_2(x),...,\pi_L(x))$.

\begin{lemma}\label{lem:pi1-continuous}
The map $\pi_L^\infty : \widetilde{X} \to (D^\infty)^L$ defined above is continuous (with the product topology in $(D^\infty)^L$).
\end{lemma}

\begin{proof}
Let $(y^n)_{n\in\mathbb N}\subseteq \widetilde{X}$ be such that $y^n\to y\in \widetilde{X}$.
For each $n\in\mathbb N$, write $
y^n=(y_1^n,y_2^n,\dots).$

If $\ell(y)\ge L$, then Corollary~\ref{convergencia em W0} implies that
$
\pi_i(y^n)\to \pi_i(y)$
for each $i\in\{1,\dots,L\},$
and hence $
\pi_L^\infty(y^n)\to \pi_L^\infty(y).$

Suppose now that $\ell(y)<L$. Again by Corollary~\ref{convergencia em W0},
$
\pi_i(y^n)\to \pi_i(y)$
for each $i\in\{1,\dots,\ell(y)\}.$
It remains to show that $
\pi_i(y^n)\to \infty=\pi_i(y)$
for each $i\in\{\ell(y)+1,\dots,L\}.$

By Corollary~\ref{convergencia em W0}, we already have $
\pi_{\ell(y)+1}(y^n)\to \infty.$
We claim that whenever $\pi_j(y^n)\to\infty$ for some $j\ge1$, then also
$\pi_{j+1}(y^n)\to\infty$.
Indeed, suppose $\pi_j(y^n)\to\infty$, but $\pi_{j+1}(y^n)\not\to\infty$ in $D^\infty$.
Then there exists a compact set $K\subseteq D$ and a subsequence
$(y^{n_p})_{p\in\mathbb N}$ such that
$
\pi_{j+1}(y^{n_p})\in K$
for all  $p\in\mathbb N.$
Since $f\colon D\to X$ is continuous, $f(K)$ is compact. Moreover, we have
$
f\bigl(\pi_{j+1}(y^{n_p})\bigr)=\pi_j(y^{n_p}).$
Hence the subsequence $(\pi_j(y^{n_p}))_{p\in\mathbb N}$ is contained in the compact set
$f(K)$, contradicting the fact that $\pi_j(y^n)\to\infty$.
Thus $\pi_{j+1}(y^n)\to\infty$, proving the claim. 

By induction, we obtain
$
\pi_i(y^n)\to\infty$
for each  $i\in\{\ell(y)+1,\dots,L\}.$
Therefore,
$
\pi_L^\infty(y^n)\to \pi_L^\infty(y),
$
as required.
\end{proof}

\begin{corollary}\label{cor:pi1-uniform}
Equip $\widetilde{X}$ with the metric from Corollary~\ref{w0 eh metrico}, and equip $D^\infty$ with any metric compatible with its one-point compactification topology. Then
$
\pi_1^\infty : \widetilde{X} \to D^\infty
$
is uniformly continuous. In particular, its restriction $
\pi_1 := \pi_1^\infty|_{\Dom(\sigma)} : \Dom(\sigma) \to D$
is uniformly continuous, where $D$ is viewed as a subspace of $D^\infty$.
\end{corollary}

\begin{proof}
By Lemma~\ref{lem:pi1-continuous}, the map $\pi_1^\infty$ is continuous. Since $\widetilde{X}$ is compact and metrizable (Corollary~\ref{cor:Xtilde-compact} and Remark~\ref{remarkXtildeultramet}), it follows that $\pi_1^\infty$ is uniformly continuous. The restriction of a uniformly continuous map to a subspace is uniformly continuous, so
$
\pi_1=\pi_1^\infty|_{\Dom(\sigma)}
$
is uniformly continuous as a map into $D^\infty$. Since its values lie in $D$, the result follows.
\end{proof}

\begin{definition}\label{def:sep-property}
Let $(X,f)$ be a Deaconu--Renault system with $\Dom(f)=D$, and let $D^\infty$ be the one-point
compactification of $D$. Fix a compatible ultrametric $d_\infty$ on $D^\infty$ and denote again by $d_\infty$
its restriction to $D$.

We say that $(X,f)$ satisfies the separation property on $D^\infty$ (with respect to $d_\infty$) if there exist
constants $R>0$ and $0<\theta<1$, and a countable family $\{W_r\}_{r\in\mathbb N}\subseteq\mathcal B$
covering $D$, such that for every $r\in\mathbb N$:
\begin{enumerate}
\item[\textnormal{(A1)}] The restriction $f|_{W_r}\colon W_r\to f(W_r)$ is a homeomorphism onto its image and
$f(W_r)\subseteq D^\infty$. Let $g_r:=(f|_{W_r})^{-1}\colon f(W_r)\to W_r$ denote the inverse branch.

\item[\textnormal{(A2)}] (Uniform contraction in $d_\infty$) The inverse branch $g_r$ is $\theta$--Lipschitz
(with respect to $d_\infty$ on $f(W_r)\subseteq D^\infty$ and on $W_r\subseteq D$), i.e.
\[
d_\infty\bigl(g_r(a),g_r(b)\bigr)\le \theta\, d_\infty(a,b)
\qquad\text{for all }a,b\in f(W_r).
\]

\item[\textnormal{(A3)}] (Uniform interior radius in $D^\infty$) For every $y\in f(W_r)$ one has
\[
B_{D^\infty}(y,R)\subseteq f(W_r).
\]
\end{enumerate}
\end{definition}

Next, we prove an auxiliary result for Deaconu--Renault systems satisfying the separation property, which will play a key role in the proof of the main shadowing theorem of this section.

\begin{lemma}\label{lem:UM-sep-inv}
Let $(X,f)$ be a Deaconu--Renault system with $\Dom(f)=D$, and let $D^\infty$ be the one-point
compactification of $D$. Assume:
\begin{enumerate}
\item[(i)] $D\subseteq f(D)\subseteq D^\infty$.

\item[(ii)] $D^\infty$ is equipped with a compatible ultrametric $d_\infty$, and we denote by the same symbol its
restriction to $D$ (and to $f(D)\subseteq D^\infty$).

\item[(iii)] $(X,f)$ satisfies the separation property on $D^\infty$ with respect to $d_\infty$
(Definition~\ref{def:sep-property}), with constants $R>0$, $0<\theta<1$, and a countable cover
$\{W_r\}_{r\in\mathbb N}\subseteq\mathcal B$ of $D$.
\end{enumerate}

Fix $0<\rho\le R$ and let $\mathcal U_\rho$ be the partition of $D$ into $d_\infty$--balls of radius $\rho$
(discarding repetitions): $
\mathcal U_\rho:=\{\,B_D(x,\rho): x\in D\,\}.$
Then:
\begin{enumerate}
\item[\textnormal{(Sep-$\rho$)}] $\mathcal U_\rho$ is clopen and $\rho$--separated: if $U\neq U'$ in $\mathcal U_\rho$
then $d_\infty(U,U')\ge \rho$.

\item[\textnormal{(Inv-$\rho$)}] For every $r\in\mathbb N$ and every atom $U\in\mathcal U_\rho$ with
$U\cap f(W_r)\neq\emptyset$, one has $U\subseteq f(W_r)$ (hence $g_r$ is defined on all of $U$) and
\[
\diam(U)\le \rho
\qquad\text{and}\qquad
\diam\bigl(g_r(U)\bigr)\le \theta\rho < \rho,
\]
where $g_r=(f|_{W_r})^{-1}\colon f(W_r)\to W_r$ is the inverse branch.
\end{enumerate}
\end{lemma}

\begin{proof}
Since $d_\infty$ is an ultrametric on $D^\infty$, every $d_\infty$--ball in $D$ is clopen in $D$
and any two balls of the same radius are either disjoint or equal. Hence $\mathcal U_\rho$ is a
clopen partition of $D$, and if $U\neq U'$ are distinct atoms then $d_\infty(U,U')\ge\rho$,
proving \textnormal{(Sep-$\rho$)}.

Let $r\in\mathbb N$ and $U\in\mathcal U_\rho$ with $U\cap f(W_r)\neq\emptyset$, and pick
$z\in U\cap f(W_r)$. Since $U$ is a $d_\infty$--ball of radius $\rho$ in the ultrametric space $(D,d_\infty)$,
we have $U=B_D(z,\rho)$. Moreover $B_D(z,\rho)=B_{D^\infty}(z,\rho)\cap D$, and since $\rho\le R$,
\[
U\subseteq B_{D^\infty}(z,\rho)\subseteq B_{D^\infty}(z,R)\subseteq f(W_r)
\]
by \textnormal{(A3)}. In particular $g_r$ is defined on all $U$ and $\diam_{d_\infty}(U)\le\rho$.

Finally, for any $a,b\in U$, by \textnormal{(A2)} we have
\[
d_\infty\bigl(g_r(a),g_r(b)\bigr)\le \theta\, d_\infty(a,b)\le \theta\,\diam_{d_\infty}(U)\le \theta\rho,
\]
so $\diam_{d_\infty}\bigl(g_r(U)\bigr)\le \theta\rho<\rho$, proving \textnormal{(Inv-$\rho$)}.
\end{proof}

We now prove the main theorem of this section, namely, that shadowing for a Deaconu--Renault system with an expanding atlas is equivalent to shadowing for its inverse-limit system.

\begin{theorem}\label{thm:transfer-shadowing}
Let $(X,f)$ be a Deaconu--Renault system with $D:=\Dom(f)$ an infinite locally compact Hausdorff space admitting a countable basis
$\mathcal B$ of compact open sets closed under finite intersections and such that $D\subseteq f(D)\subseteq D^\infty$.
Let $(E,\alpha_f)$ be the associated Deaconu--Renault system from
Proposition~\ref{prop:shift-DR-correct} and assume that $\vec 0\in E$.
Fix a compatible ultrametric $d_\infty$ on $D^\infty$, and suppose  that $(X,f)$ satisfies the
separation property on $D^\infty$ with respect to $d_\infty$, as in Definition \ref{def:sep-property}.

Then $(E,\alpha_f)$ has the shadowing property if and only if $(D^\infty,f)$, with
$\Dom(f)=D$, has the shadowing property.
\end{theorem}

\begin{proof}
Suppose first that $(E,\alpha_f)$ has the shadowing property.

Let $\varepsilon>0$ be given. We will find $\delta>0$ such that every $\delta$--pseudo-orbit
$(x_i)_{i\ge0}\subseteq D$ for $f$ (measured with $d_\infty$) is $\varepsilon$--shadowed by an $f$--orbit.

We keep the ultrametric $d$ on $W_0$ defined in Section~\ref{W0 section} via the enumeration
$\mathcal P=\{p_1,p_2,\dots\}$ and the generalized cylinders $Z[p_j,\emptyset]$.

\bigskip
\noindent {\it Step 1: Shadowing upstairs and uniform continuity of the coordinate projection.}

Let $\pi\colon E\to D^\infty$ be the first-coordinate map (with $\pi(\vec 0)=\infty$), i.e.
$\pi(y_1y_2\cdots)=y_1$ for $\ell(y)\ge1$.
By Lemma~\ref{lem:pi1-continuous}, $\pi$ is continuous. Since $\vec{0}\in E$, then by Proposition \ref{prop:shift-DR-correct} $E$ is compact, and so $\pi$ is uniformly continuous.
Choose $n\in\mathbb N$ such that $2^{-n}\le \varepsilon$ and set $\eta:=2^{-(n+1)}$.
Then there exists $k\in\mathbb N$ such that
\begin{equation}\label{eq:pi-uc}
d(u,v)<2^{-k}\ \Longrightarrow\ d_\infty\bigl(\pi(u),\pi(v)\bigr)<\eta.
\end{equation}

Define, for each $m\in\N$, \[
\mathcal V_m:=\{\,B_E(z,2^{-m}): z\in E\,\}\quad (m\in\mathbb N)
\]

(discarding repetitions, and restricting to $\Dom(\alpha_f)=E\setminus\{\vec 0\}$) 
which, by Proposition~\ref{prop:ball-tame-defining}, is a tame defining sequence. 

Since $(E,\alpha_f)$ has shadowing, applying Proposition~\ref{prop:shadowing-char} to $\{\mathcal{V}_m\}_{m\in \N}$ gives an index $l\ge k$ such that every $\mathcal V_l$--pseudo-orbit for $\alpha_f$
is $\mathcal V_k$--shadowed.

\bigskip
\noindent {\it Step 2: A finite family of basic sets detected by the first $l$ cylinders.}

Let $L:=\max\{\ell(p_j):1\le j\le l\}$ and let $\mathcal F_l\subseteq\mathcal B$ be the finite family of compact open basic sets that appear among the first $l$ words $p_1,\dots,p_l$ (i.e.\ each $p_j$ has the form $p_j=B_j^1\times...\times B_j^{\ell(p_j)}$ and $\mathcal{F}_l$ is the set of all the $B_t^j$ occurring there).
For each $B\in\mathcal F_l$, since $B$ is compact and $D\setminus B$ is closed in the metric space $(D,d_\infty)$,
we have $\dist(B,D\setminus B)>0$ (with the metric $d_\infty$). Set
\[
\rho_l^{\mathcal F}:=\min_{B\in\mathcal F_l}\dist(B,\,D\setminus B)\ >0.
\]

Choose $\rho>0$ so small that
\begin{equation}\label{eq:rho-small}
0<\rho\le R
\qquad\text{and}\qquad
\rho<\rho_l^{\mathcal F}, 
\end{equation}
where $R$ is as in Definition \ref{def:sep-property}.

Let $\mathcal U_\rho$ be the radius-$\rho$ ball partition of $D$ from Lemma~\ref{lem:UM-sep-inv}.
Set
\[
\delta:=\rho/2.
\]

\bigskip
\noindent {\it Step 3: Lifting an $f$--pseudo-orbit to an $\alpha_f$--pseudo-orbit (finite depth).}

Let $(x_i)_{i\ge0}\subseteq D$ be a $\delta$--pseudo-orbit for $f$:
\[
d_\infty\bigl(f(x_i),x_{i+1}\bigr)<\delta\qquad(i\ge0).
\]
Since $\delta<\rho$ and $\mathcal U_\rho$ is $\rho$--separated, for each $i$ there is a unique atom
$U_i^{(0)}\in\mathcal U_\rho$ such that
\begin{equation}\label{eq:Ui0}
f(x_i)\in U_i^{(0)}\quad\text{and}\quad x_{i+1}\in U_i^{(0)}.
\end{equation}

Choose $y_0\in\Dom(\alpha_f)$ with $\pi(y_0)=x_0$ (exists since $D\subseteq f(D)$) and write
$y_0=(y_{0,1},y_{0,2},\dots)$ with $y_{0,1}=x_0$ and $f(y_{0,t+1})=y_{0,t}$ for all $t\ge1$.

Inductively, assume $y_i\in\Dom(\alpha_f)$ is chosen and write
\[
y_i=(y_{i,1},y_{i,2},\dots)=(x_i,y_{i,2},y_{i,3},\dots),
\qquad f(y_{i,t+1})=y_{i,t}.
\]
Let $a_t$ denote the $t$-th coordinate of $\alpha_f(y_i)$:
\[
a_0=f(x_i),\quad a_1=x_i,\quad a_2=y_{i,2},\ \dots,\ a_{L-1}=y_{i,L-1}.
\]

We define the first $L$ coordinates of $y_{i+1}$ by a backward recursion.
Set $b_0:=x_{i+1}$. For $t=0,1,\dots,L-2$ do:
\begin{itemize}
\item Let $U_i^{(t)}\in\mathcal U_\rho$ be the unique atom containing both $a_t$ and $b_t$
(existence/uniqueness follows from $d_\infty(a_t,b_t)<\rho$ and \textnormal{(Sep-$\rho$)}; for $t=0$ this is
$U_i^{(0)}$ by \eqref{eq:Ui0}).

\item Choose $r_{i,t}\in\mathbb N$ such that $a_{t+1}\in W_{r_{i,t}}$ (possible since $\{W_r\}$ covers $D$).
Then $a_t=f(a_{t+1})\in f(W_{r_{i,t}})$, so $U_i^{(t)}\cap f(W_{r_{i,t}})\neq\emptyset$.
By Lemma~\ref{lem:UM-sep-inv} we have $U_i^{(t)}\subseteq f(W_{r_{i,t}})$, hence $g_{r_{i,t}}$ is defined on $U_i^{(t)}$.

\item Define
\[
b_{t+1}:=g_{r_{i,t}}(b_t).
\]
Since $a_{t+1}\in W_{r_{i,t}}$ and $f(a_{t+1})=a_t$, we also have $a_{t+1}=g_{r_{i,t}}(a_t)$.
Using \textnormal{(A2)} and the fact that $a_t,b_t\in U_i^{(t)}$ we obtain
\[
d_\infty(a_{t+1},b_{t+1})
=d_\infty\bigl(g_{r_{i,t}}(a_t),g_{r_{i,t}}(b_t)\bigr)
\le \theta\, d_\infty(a_t,b_t)
<\theta\,\rho<\rho.
\]
Therefore $a_{t+1}$ and $b_{t+1}$ belong to the same (unique) atom $U_i^{(t+1)}\in\mathcal U_\rho$.
\end{itemize}

Define $(y_{i+1})_t:=b_{t-1}$ for $1\le t\le L$ (so the first coordinate is $x_{i+1}=b_0$), and extend the finite
backward string to an infinite backward $f$--orbit (possible since $D\subseteq f(D)$) to obtain
$y_{i+1}\in\Dom(\alpha_f)$.

\bigskip
\noindent {\it Claim. $(y_i)_{i\ge0}$ is a $\mathcal V_l$--pseudo-orbit for $\alpha_f$, i.e.
$d(\alpha_f(y_i),y_{i+1})<2^{-l}$ for all $i$.}

\smallskip
\noindent {\it Proof of the claim:}
Fix $j\le l$ and write $p_j=B_1\times\cdots\times B_m$ with $m=\ell(p_j)\le L$.
From the definition of $Z[p_j, \emptyset]$, membership in $Z[p_j,\emptyset]$ depends only on the first $m$ coordinates.

For each $t=0,1,\dots,m-1$, the $t$-th coordinate of $\alpha_f(y_i)$ is $a_t$ and the $t$-th coordinate of $y_{i+1}$
is $b_t$. By construction $a_t$ and $b_t$ lie in the same atom of $\mathcal U_\rho$, hence
$d_\infty(a_t,b_t)\le\diam(U_i^{(t)})\le\rho<\rho_l^{\mathcal F}$.
Therefore $a_t$ and $b_t$ have the same membership in every $B\in\mathcal F_l$, and in particular
\[
a_t\in B_{t+1}\iff b_t\in B_{t+1}\qquad(0\le t\le m-1).
\]
Hence $\alpha_f(y_i)\in Z[p_j,\emptyset]\iff y_{i+1}\in Z[p_j,\emptyset]$ for all $j\le l$.
By the definition of the ultrametric $d$ on $W_0$, this implies $d(\alpha_f(y_i),y_{i+1})<2^{-l}$.
This proves the claim.

\bigskip
\noindent {\it Step 4: Shadow upstairs and project down.}

By shadowing on $(E,\alpha_f)$ (see the last line of Step 1), there exists $z\in\Dom(\alpha_f)$ such that
\[
d(\alpha_f^i(z),y_i)<2^{-k}\qquad\forall i\ge0.
\]
Applying \eqref{eq:pi-uc} yields
\[
d_\infty\bigl(\pi(\alpha_f^i(z)),\pi(y_i)\bigr)<\eta=2^{-(n+1)}\qquad\forall i\ge0.
\]
Since $\pi(y_i)=x_i$ and $\pi(\alpha_f^i(z))=f^i(\pi(z))$, letting $x:=\pi(z)\in D$ we obtain
\[
d_\infty\bigl(f^i(x),x_i\bigr)<2^{-(n+1)}\le \varepsilon\qquad\forall i\ge0,
\]
so $(x_i)_{i\geq 0}$ is $\varepsilon$--shadowed by the orbit of $x$. This proves shadowing for $(D^\infty,f)$.

\medskip

 Now, suppose that $(D^\infty,f)$ has the shadowing property. We will show that $(E,\alpha_f)$ has the shadowing property using the characterization of shadowing via ball partitions (equivalently, via a tame defining
sequence) for ultrametric systems.

For each $m\geq 1$ let $\mathcal V_m:=\{B_E(z,2^{-m}):z\in E\}$ (discarding repetitions) be the ball
partition of $(E,d)$; then $(\mathcal V_m)_{m\ge1}$ is a tame defining sequence on $E$, by Proposition \ref{prop:ball-tame-defining}.
Fix $k\in\mathbb N$. It suffices to show that there exists $m\ge k$ such that every $\mathcal V_m$--pseudo-orbit for $\alpha_f$
(in $\Dom(\alpha_f)$) is $\mathcal V_k$--shadowed.

\bigskip
\noindent {\it Step 1: Choose a scale $\rho$ in $D^\infty$ controlled by the first $k$ cylinders.}

Let
\[
L:=\max\{\ell(p_j):1\le j\le k\},
\]
and let $\mathcal F_k\subseteq\mathcal B$ be the finite family of basis sets that appear among
$p_1,\dots,p_k$ (where $\mathcal{P}=\{p_1,p_2,...\}$ is the enumeration fixed in the beginning of this proof).
Set
\[
\rho_k^{\mathcal F}:=\min_{B\in\mathcal F_k}\dist(B,\,D\setminus B)>0.
\]
Choose $0<\rho\le R$ (where $R$ is as in Definition \ref{def:sep-property}) such that
\begin{equation}\label{eq:rho-choice-conv-rev}
\rho<\rho_k^{\mathcal F}.
\end{equation}
Then any subset of $D$ of $d_\infty$--diameter $\le \rho$ is either contained in $B$ or disjoint from $B$,
for every $B\in\mathcal F_k$.

Let $\mathcal U_\rho$ be the partition of $D$ into $d_\infty$--balls of radius $\rho$
(discarding repetitions). Notice that, by Lemma~\ref{lem:UM-sep-inv}, $\mathcal U_\rho$ is clopen and $\rho$--separated, and
for every $r$ and every $U\in\mathcal U_\rho$ with $U\cap f(W_r)\neq\emptyset$ we have $U\subseteq f(W_r)$ and
\begin{equation}\label{eq:inv-rho-conv-rev}
\diam(g_r(U))\le \theta\rho<\rho.
\end{equation}

\bigskip
\noindent {\it Step 2: Obtain a downstairs shadowing scale $\delta_f$ for tolerance $\rho/2$.}

Since $(D^\infty,f)$ has the shadowing property, apply it with tolerance $\rho/2$ to obtain $\delta_f>0$ such that:
\begin{equation}\label{eq:f-shadow-rev}
\text{every $\delta_f$--pseudo-orbit $(x_i)_{i\ge0}\subseteq D$ for $f$ is $(\rho/2)$--shadowed by some orbit in $D$.}
\end{equation}

\bigskip
\noindent {\it Step 3: Choose $m$ so that $d$--closeness in $E$ controls the first $L$ coordinates in $d_\infty$.}

For $t\ge1$, let $\pi_t:E\to D^\infty$ denote the $t$--th coordinate map (with the convention that if a word has
length $<t$, then $\pi_t$ takes value $\infty\in D^\infty$). Consider
\[
\Pi_L:E\longrightarrow (D^\infty)^L,\qquad \Pi_L(y):=(\pi_1(y),\dots,\pi_L(y)),
\] which is a continuous map, from Lemma \ref{lem:pi1-continuous}. Since $E$ is compact and $\Pi_L$ is continuous, $\Pi_L$ is uniformly continuous. Hence there exists $m_0\in\mathbb N$
such that
\begin{equation}\label{eq:PiL-uc-rev}
d(u,v)<2^{-m_0}\ \Longrightarrow\ d_\infty(\pi_t(u),\pi_t(v))<\min\{\rho/2,\delta_f\}
\ \ \text{for all }1\le t\le L.
\end{equation}
Now fix $m:=\max\{k,m_0\},$ where $k$ is the fixed natural number in the beginning of this proof.

\bigskip
\noindent {\it Step 4: Project an $\mathcal V_m$--pseudo-orbit in $E$ to a $\delta_f$--pseudo-orbit in $D$.}

Let $(y_i)_{i\ge0}\subseteq\Dom(\alpha_f)$ be a $\mathcal V_m$--pseudo-orbit for $\alpha_f$, i.e. $
d\bigl(\alpha_f(y_i),y_{i+1}\bigr)<2^{-m}$, $i\ge 0$. 
Set $x_i:=\pi_1(y_i)\in D$. Applying \eqref{eq:PiL-uc-rev} to $\alpha_f(y_i)$ and $y_{i+1}$ at coordinate $t=1$ yields
\[
d_\infty\bigl(\pi_1(\alpha_f(y_i)),\,\pi_1(y_{i+1})\bigr)<\delta_f.
\]
But $\pi_1(\alpha_f(y_i))=f(\pi_1(y_i))=f(x_i)$, so $
d_\infty\bigl(f(x_i),x_{i+1}\bigr)<\delta_f$, $i\ge0,$
i.e.\ $(x_i)_{i\geq 0}$ is a $\delta_f$--pseudo-orbit for $f$ in $D$. By \eqref{eq:f-shadow-rev}, there exists $x\in D$ with
$f^i(x)$ defined for all $i\ge0$ and
\begin{equation}\label{eq:shadow-base-rev}
d_\infty\bigl(f^i(x),x_i\bigr)<\rho/2\qquad(i\ge0).
\end{equation}

\bigskip
\noindent {\it Step 5: Build a lift $z\in\Dom(\alpha_f)$ of $x$ and compare $\alpha_f^i(z)$ with $y_i$.}

Write $y_0=y_{0,1}y_{0,2}\cdots$.
For each $t\ge1$ such that $\pi_{t+1}(y_0)\in D$ (i.e.\ the $(t+1)$-st coordinate is an actual point of $D$),
choose an index $r_t$ with $y_{0,t+1}\in W_{r_t}$, so that $y_{0,t}=f(y_{0,t+1})\in f(W_{r_t})$ (where the sets $W_r$ are as in Definition \ref{def:sep-property}).

Define $z=z_1z_2\cdots\in D_0$ recursively by
\[
z_1:=x,\qquad z_{t+1}:=g_{r_t}(z_t),
\]
whenever $r_t$ has been chosen and $g_{r_t}(z_t)$ is defined. We claim (by induction on $t$) that as long as
$r_t$ is chosen, this recursion is well-defined and satisfies
\begin{equation}\label{eq:z-close-y0-rev}
d_\infty(z_t,y_{0,t})<\rho/2\qquad\text{for all such }t.
\end{equation}
Indeed, for $t=1$, \eqref{eq:shadow-base-rev} gives
$d_\infty(z_1,y_{0,1})=d_\infty(x,x_0)<\rho/2\le R$. Since $y_{0,1}\in f(W_{r_1})$ and
$B_{D^\infty}(y_{0,1},R)\subseteq f(W_{r_1})$ by (A3), we have $z_1\in f(W_{r_1})$, so $z_2=g_{r_1}(z_1)$ is defined.
Then (A2) gives
\[
d_\infty(z_2,y_{0,2})
=d_\infty\bigl(g_{r_1}(z_1),g_{r_1}(y_{0,1})\bigr)
\le \theta\, d_\infty(z_1,y_{0,1})
<\rho/2.
\]
The inductive step is identical: if $d_\infty(z_t,y_{0,t})<\rho/2\le R$ and $y_{0,t}\in f(W_{r_t})$, then (A3) gives
$z_t\in f(W_{r_t})$, so $z_{t+1}$ is defined, and (A2) gives $d_\infty(z_{t+1},y_{0,t+1})<\rho/2$.

Extending (if necessary) the finite backward string thus obtained to an infinite backward $f$--orbit (using $D\subseteq f(D)$),
we obtain $z\in E$ with $z\neq\vec 0$, hence $z\in\Dom(\alpha_f)$.

\bigskip
\noindent {\it Step 6: Show that $z$ $\mathcal V_k$--shadows $(y_i)_{i\geq 0}$.}

We first prove that for every $i\ge0$ and every $1\le t\le L$,
\begin{equation}\label{eq:coord-close-rev}
d_\infty\bigl(\pi_t(\alpha_f^i(z)),\,\pi_t(y_i)\bigr)<\rho.
\end{equation}

Fix $i\ge0$.

{\it Case 1: $t\le i+1$.}

Then $\pi_t(\alpha_f^i(z))=f^{\,i-t+1}(z_1)=f^{\,i-t+1}(x)$. On the other hand, since $(y_i)_{i\geq 0}$ is a $\mathcal V_m$--pseudo-orbit,
applying \eqref{eq:PiL-uc-rev} to each pair $(\alpha_f(y_s),y_{s+1})$ at coordinate $t$ gives (for $2\le t\le L$)
\[
d_\infty\bigl(\pi_t(y_{s+1}),\,\pi_{t-1}(y_s)\bigr)
=d_\infty\bigl(\pi_t(y_{s+1}),\,\pi_t(\alpha_f(y_s))\bigr)
<\rho/2.
\]
Iterating this inequality $t-1$ times and using the ultrametric triangle inequality yields
\[
d_\infty\bigl(\pi_t(y_i),\,x_{i-t+1}\bigr)<\rho/2
\qquad\text{whenever }t\le i+1\text{ and }t\le L.
\]
Combining with \eqref{eq:shadow-base-rev} (at index $i-t+1$) gives
\[
d_\infty\bigl(\pi_t(\alpha_f^i(z)),\,\pi_t(y_i)\bigr)
\le \max\Bigl\{d_\infty\bigl(f^{\,i-t+1}(x),x_{i-t+1}\bigr),\, d_\infty\bigl(x_{i-t+1},\pi_t(y_i)\bigr)\Bigr\}
<\rho/2<\rho.
\]

{\it Case 2: $t>i+1$ (so $t-i\ge2$ and $t\le L$).}

Then $\pi_t(\alpha_f^i(z))=z_{t-i}$. By construction \eqref{eq:z-close-y0-rev},
$d_\infty(z_{t-i},y_{0,t-i})<\rho/2$. Also, iterating the coordinate-shift inequalities as above gives $
d_\infty\bigl(\pi_t(y_i),\,y_{0,t-i}\bigr)<\rho/2$, $t\le L,$
so again
\[
d_\infty\bigl(\pi_t(\alpha_f^i(z)),\,\pi_t(y_i)\bigr)
\le \max\Bigl\{d_\infty(z_{t-i},y_{0,t-i}),\, d_\infty(y_{0,t-i},\pi_t(y_i))\Bigr\}
<\rho/2<\rho.
\]
This proves \eqref{eq:coord-close-rev}.

Now let $j\le k$ ($k$ is fixed in the beginning of this proof) and write $p_j=B_1\times\cdots\times B_m$ with $m=\ell(p_j)\le L$.
By \eqref{eq:coord-close-rev} and \eqref{eq:rho-choice-conv-rev}, for each $1\le s\le m$ the points
$(\alpha_f^i(z))_s$ and $(y_i)_s$ lie in the same side of every $B\in\mathcal F_k$, hence in particular
$
(\alpha_f^i(z))_s\in B_s\iff (y_i)_s\in B_s.$
Therefore $\alpha_f^i(z)\in Z[p_j,\emptyset]\iff y_i\in Z[p_j,\emptyset]$ for all $j\le k$.
By the definition of the ultrametric $d$ on $W_0$, this implies
\[
d\bigl(\alpha_f^i(z),y_i\bigr)<2^{-k}\qquad(i\ge0),
\]
so $(y_i)_{i\geq 0}$ is $\mathcal V_k$--shadowed by the $\alpha_f$--orbit of $z$.
This proves shadowing for $(E,\alpha_f)$.
\end{proof}

The requirement in Theorem~\ref{thm:transfer-shadowing} that the compactification $D^\infty$ admit a compatible ultrametric is not restrictive, as we show below.

\begin{proposition}\label{prop:one-point-compactification-ultrametric}
Let $X$ be a noncompact locally compact Hausdorff space with a countable basis
$\mathcal B$ of compact open sets closed under finite intersections, and let $X^\infty$
be the one-point compactification of $X$.

Then $X^\infty$ is a compact zero-dimensional metrizable space and, moreover,
$X^\infty$ admits a compatible complete ultrametric.

\end{proposition}

\begin{proof}
Since $X^\infty$ is compact
Hausdorff and second countable, it is metrizable.
It remains to show that $X^\infty$ is zero-dimensional. Let us prove that it has a basis of
clopen sets.

First, if $x\in X$, then by hypothesis there exists $B\in\mathcal B$ with $x\in B$, and every
such $B$ is clopen in $X$. Since $B$ is compact in $X$, it is also closed in $X^\infty$; hence
$B$ is clopen in $X^\infty$. Therefore points of $X$ have a clopen neighborhood basis in $X^\infty$.

Now consider the point $\infty\in X^\infty$. A neighborhood basis of $\infty$ is given by sets of the form
$
U_K:=X^\infty\setminus K,$
where $K\subseteq X$ is compact. Fix such a compact $K$. Since $\mathcal B$ is a basis of $X$,
for each $x\in K$ choose $B_x\in\mathcal B$ such that $x\in B_x$. By compactness of $K$, there exist
$x_1,\dots,x_m\in K$ such that
\[
K\subseteq B_{x_1}\cup\cdots\cup B_{x_m}.
\]
Set $
C:=B_{x_1}\cup\cdots\cup B_{x_m}.$ 
Then $C$ is compact and open in $X$, hence clopen in $X^\infty$. Moreover,
$
\infty\in X^\infty\setminus C\subseteq X^\infty\setminus K=U_K.$
Thus $\infty$ also has a clopen neighborhood basis.

Therefore $X^\infty$ is zero-dimensional. Since it is also compact and metrizable, it is Polish.
By \cite[Proposition~3.1.3]{DGS}, $X^\infty$ admits a complete defining sequence, and then by
\cite[Proposition~3.1.6]{DGS} the associated metric is a compatible complete ultrametric.
\end{proof}

\subsection{Examples}

In this section we present examples of spaces and Deaconu--Renault systems satisfying the hypothesis of Section~2, Section~5, and of Theorem~\ref{thm:transfer-shadowing}.

Assume that $(X,d)$ is a compact ultrametric space and let $D\subseteq X$ be a nonempty open subset. We first observe that such subsets satisfy the standing hypotheses of Section~\ref{W0 section}.

The following proposition establishes the uniform equivalence of compatible ultrametrics on certain open subsets of a compact space, a fact that will be used repeatedly in the examples below.

\begin{proposition}\label{lem:open-subset-ultra}
Let $(X,d)$ be a compact ultrametric space, and let $D \subseteq X$ be a nonempty
open infinite subset. Then $D$ is an infinite locally compact Hausdorff space
which admits a countable basis $\mathcal B$ of compact open sets closed under
finite intersections.
\end{proposition}

\begin{proof}
Since $X$ is compact Hausdorff and $D$ is open in $X$, it follows that $D$ is
locally compact Hausdorff. As $X$ is metrizable, it is second countable, and
hence so is $D$.

We first construct a suitable basis on $X$. Fix a countable dense subset
$\{x_j\}_{j\in\N} \subseteq X$ and consider positive rational radii
$\{r_k\}_{k\in\N}$. For each $j,k$ let
\[
B_{j,k} := \overline B(x_j,r_k)
  = \{ y \in X : d(x_j,y) \le r_k\}.
\]
In an ultrametric space every closed ball is also open, so each $B_{j,k}$ is
clopen. Since $X$ is compact, each $B_{j,k}$ is also compact. The family
\[
\mathcal C := \{ B_{j,k} : j,k \in \N \}
\]
is countable and forms a basis for the topology of $X$. Moreover, in an
ultrametric space the intersection of two balls is either empty or itself a
ball, so $\mathcal C$ is closed under finite intersections (up to possibly
discarding the empty set).

Now define
\[
\mathcal B := \{ B \in \mathcal C : B \subseteq D\}.
\]
Each $B \in \mathcal B$ is open in $X$, hence open in $D$, and compact in $X$,
hence compact in $D$. Thus $\mathcal B$ is a family of compact open subsets of
$D$. It is clearly countable and closed under finite intersections (again,
ignoring the empty set).

Clearly, $\mathcal B$ is a basis for the topology of $D$.
\end{proof}

\begin{proposition}\label{prop:uniform-equivalence}
Let $(X,d)$ be a compact ultrametric space and let $p\in X$.  
Set $D := X\setminus\{p\}$, which is an open subset of $X$.  
Let $\widetilde{D}$ denote the one-point compactification of $D$, and let
$\widetilde d$ be any ultrametric on $\widetilde{D}$ that induces the 
one-point compactification topology.  

Then $\widetilde{D}$ is homeomorphic to $X$, and under this identification,
the ultrametrics $d$ and $\widetilde d$ (transported to $X$) are uniformly equivalent 
on $X$.
\end{proposition}

\begin{proof}
By \cite[Theorem~29.1]{munkres2edition}, the one-point compactification of 
$D=X\setminus\{p\}$ is homeomorphic to $X$.  
More precisely, there exists a homeomorphism $
h : \widetilde{D} \longrightarrow X$
such that $h|_{D} = \mathrm{id}_{D}$ and $h(\infty)=p$.

Define a new metric $\widetilde{d_X}$ on $X$ by transporting $\widetilde d$ through $h$:
\[
\widetilde d_X(x,y) := \widetilde d\big(h^{-1}(x), h^{-1}(y)\big), \qquad x,y\in X.
\]
Since $h$ is a homeomorphism, $\widetilde d_X$ is a metric that induces 
the same topology on $X$ as $d$.

Because $(X,d)$ and $(X,\widetilde d_X)$ are compact metric spaces with the 
same topology, their metrics are uniformly equivalent:  
for every $\varepsilon>0$ there exists $\delta>0$ such that
\[
d(x,y)<\delta \ \Longrightarrow\ \widetilde d_X(x,y)<\varepsilon,
\quad\text{and conversely}.
\]

Finally, since $h$ fixes $D$ pointwise, the restriction of $\widetilde d_X$ to $D$ 
coincides with the restriction of $\widetilde d$ to $D$.  
Thus, under the identification of $\widetilde{D}$ with $X$ via $h$,
the metrics $d$ and $\widetilde d$ are uniformly equivalent on $X$, and hence 
also on $D$.
\end{proof}

\begin{example}
    The proposition above applies, for instance, to any OTW subshift, by taking $p$ to be the empty word.
\end{example}

\begin{remark}\label{rmk:metric-compactification}
Since $D$ is a metric, locally compact, $\sigma$-compact space, its one-point
compactification $D^\infty$ is metrizable. Fix a metric $d^\infty$ on $D^\infty$
compatible with its topology. Then the restriction of $d^\infty$ to $D$ generates
the original topology of $D$ induced by $d$. 

\end{remark}
We conclude the paper with two examples of Deaconu--Renault systems satisfying the hypotheses of Theorem~\ref{thm:transfer-shadowing}.

\paragraph{Example (variable-length shift on $X\setminus\{1^\infty\}$).}
Let $X:=\{0,1\}^{\mathbb N}$ with the product topology and the standard ultrametric
\[
d(x,y)=
\begin{cases}
0,&x=y,\\
2^{-N},&N=\min\{n\ge 0:\ x_n\neq y_n\}.
\end{cases}
\]
Let $p:=1^\infty=(1,1,1,\dots)$ and set
$
D:=X\setminus\{p\}.$
Then $D$ is locally compact and zero-dimensional, hence it admits a countable basis of compact open sets
closed under finite intersections (e.g.\ cylinder sets contained in $D$).

For each $n\ge 0$, define the cylinder
\[
W_n:=Z(1^n0)=\{x\in X:\ x_0=\cdots=x_{n-1}=1,\ x_n=0\}.
\]
Then each $W_n$ is compact open, contained in $D$, and
\[
D=\bigsqcup_{n\ge 0} W_n
\quad\text{(disjoint union)}.
\]

Define $f\colon D\to X$ by
\[
f(x)=\sigma^{n+1}(x)\qquad\text{if }x\in W_n,
\]

where $\sigma$ is the left shift on $X$.

Since each restriction $f|_{W_n}=\sigma^{n+1}|_{W_n}$ is a homeomorphism onto $X$, the map $f$ is a local
homeomorphism. Thus $(X,f)$ is a Deaconu--Renault system with $\Dom(f)=D$.

\bigskip
{\it Compactification and metric.}

The one-point compactification of $D=X\setminus\{p\}$ is canonically homeomorphic to $X$ by adding back the point $p$. 
Hence we identify $
D^\infty:=X$, $
d_\infty:=d,$ 
so $d_\infty$ is a compatible ultrametric on $D^\infty$ and its restriction to $D$ is the metric on $D$.
Moreover, $f(D)=X=D^\infty$ (indeed, for any $y\in X$ we have $0y\in W_0\subseteq D$ and $f(0y)=y$), so $
D\subseteq f(D)\subseteq D^\infty$ 
as required in Lemma~\ref{lem:UM-sep-inv}(i).

\bigskip
{\it Verification of A1--A3 (Definition~\ref{def:sep-property}).}

The family $\{W_n\}_{n\ge 0}\subseteq\mathcal B$ is a countable cover of $D$.
For each $n\ge 0$:
\begin{itemize}
\item[\textnormal{(A1)}] $f|_{W_n}\colon W_n\to f(W_n)=X$ is a homeomorphism and $f(W_n)\subseteq D^\infty$.
Its inverse branch $g_n\colon X\to W_n$ is the prefixing map
\[
g_n(y)=1^n0\,y\in W_n.
\]

\item[\textnormal{(A2)}] $g_n$ is $\theta$--Lipschitz with the uniform constant $\theta=\frac12$:
for all $a,b\in X$,
\[
d_\infty\bigl(g_n(a),g_n(b)\bigr)=2^{-(n+1)}\,d_\infty(a,b)\le \tfrac12\,d_\infty(a,b).
\]

\item[\textnormal{(A3)}] Since $f(W_n)=X=D^\infty$ for every $n$, the uniform interior radius holds for any $R>0$:
\[
\forall y\in f(W_n)=X,\qquad B_{D^\infty}(y,R)\subseteq X=f(W_n).
\]
For instance, take $R=1$.
\end{itemize}

\smallskip
\noindent
Therefore $(X,f)$ satisfies the separation property on $D^\infty$ with respect to $d_\infty$,
and Lemma~\ref{lem:UM-sep-inv} applies (with the ball partitions in $(D,d_\infty)$).

Let $(E,\alpha_f)$ be the Deaconu--Renault system associated to $(X,f)$, as in Section~\ref{tornado}. We claim that $\vec{0}\in E$, so that Theorem~\ref{thm:transfer-shadowing} applies in this example. Let $\widetilde{X}$ be as in Section~\ref{tornado}. By Proposition~\ref{prop:shift-DR-correct}, it suffices to show that $\widetilde{X}$ contains an element of length $1$.

Set $w=10^\infty\in D$. For each $n\in\mathbb N$ and $m\in\mathbb N$, define $x_n^m:=(1^n0)^m w.$
Then $f(x_n^{m+1})=x_n^m$, $m\ge 0.$
Hence $
x_n:=x_n^0x_n^1x_n^2\cdots \in \widetilde{X}$
is an element of infinite length. Moreover, for every $n\in\mathbb N$ we have $x_n^0=w$, so the sequence $(x_n^0)_{n\in\mathbb N}$ is constant and therefore converges to $w$ in $D$. On the other hand, $(x_n^1)_{n\in\mathbb N}$ converges to $\infty$ in $D^\infty$ (equivalently, to $1^\infty$ in $X$). It then follows from Corollary~\ref{convergencia em W0} that $(x_n)_{n\in\mathbb N}$ converges to $w$ in $D_0$. Therefore $w\in \widetilde{X}$.
Since $\ell(w)=1$, Proposition~\ref{prop:shift-DR-correct} yields $\vec{0}\in E$. Consequently, Theorem~\ref{thm:transfer-shadowing} applies to this example.

\paragraph{Example: First return map on a Cantor set}
Let $X:=\{0,1\}^{\mathbb N}$ be the Cantor space with the standard ultrametric: $d(x,y)=2^{-N} \text{ if } x\neq y,\text{ where } N=\min\{n\ge 0: x_n\neq y_n\}.$

Let $\sigma\colon X\to X$ be the left shift and fix the clopen set
\[
A:=Z(0)=\{x\in X:\ x_0=0\}.
\]
Define the (first) return time $\tau\colon A\to\mathbb N$ by
\[
\tau(x):=\min\{n\ge 1:\ \sigma^n(x)\in A\}.
\]
(If $x\in A$ has the form $x=00u$ with $u\in \{0,1\}^\N$ then $\tau(x)=1$ and if $x=0\,1^k\,0\,u$ with $k\ge 1,\ u\in\{0,1\}^{\mathbb N}$, then $\tau(x)=k+1$).

Let
\[
D:=A\setminus\{0\,1^\infty\},
\]
so $\tau$ is finite on $D$ and $D$ is open in $A$ (hence locally compact) but $D\neq X$.

Define the return map $f\colon D\to A$ by
\[
f(x):=\sigma^{\tau(x)}(x).
\]
Then $f$ is a local homeomorphism (hence $(X,f)$ is a Deaconu--Renault system with $\Dom(f)=D$):
indeed, $D$ decomposes as the disjoint union of compact-open sets
\[
W_k:=Z(0\,1^k\,0)\qquad (k\ge 0),
\]
and on each $W_k$ we have $\tau\equiv k+1$, so $
f|_{W_k}=\sigma^{k+1}|_{W_k}\colon W_k\to A$
is a homeomorphism onto $A$.

\bigskip
{\it Compactification and metric.}

Since $D=A\setminus\{0\,1^\infty\}$, its one-point compactification $D^\infty$ is naturally identified with $A$
(by adding back the missing point $0\,1^\infty$). Under this identification we take $
D^\infty:=A$, 
$
d_\infty:=d|_A,$
which is a compatible ultrametric on $D^\infty$; its restriction to $D$ is the metric on $D$ used below.
Moreover, $f(D)=A=D^\infty$, so
$
D\subseteq f(D)\subseteq D^\infty
$
as required in Lemma~\ref{lem:UM-sep-inv}(i).

\bigskip
{\it Verification of A1--A3 (Definition~\ref{def:sep-property}).}

The family $\{W_k\}_{k\ge 0}\subseteq\mathcal B$ is a countable cover of $D$.
For each $k\ge 0$:
\begin{itemize}
\item[\textnormal{(A1)}] $f|_{W_k}\colon W_k\to f(W_k)=A$ is a homeomorphism.
Its inverse branch $g_k\colon A\to W_k$ is the prefixing map
\[
g_k(y)=0\,1^k\,y.
\]
In particular $f(W_k)=A\subseteq D^\infty$.

\item[\textnormal{(A2)}] $g_k$ is $\theta$--Lipschitz on $f(W_k)=A$ with the uniform constant $\theta=\frac12$:
for all $a,b\in A$,
\[
d_\infty\bigl(g_k(a),g_k(b)\bigr)=2^{-(k+1)}\,d_\infty(a,b)\le \tfrac12\,d_\infty(a,b).
\]

\item[\textnormal{(A3)}] Since $f(W_k)=A=D^\infty$ for every $k$, the uniform interior radius holds for any $R>0$:
\[
\forall y\in f(W_k)=A,\qquad B_{D^\infty}(y,R)\subseteq A=f(W_k).
\]
For instance, take $R=1$.
\end{itemize}

\smallskip
\noindent
Therefore $(X,f)$ satisfies the separation property on $D^\infty$ with respect to $d_\infty$,
and Lemma~\ref{lem:UM-sep-inv} applies (with the ball partitions in $(D,d_\infty)$).

As in the previous example, we show that Theorem~\ref{thm:transfer-shadowing} applies in this case as well. Let $(E,\alpha_f)$ be the Deaconu--Renault system associated to $(X,f)$, as in Section~\ref{tornado}. By Proposition~\ref{prop:shift-DR-correct}, it suffices to show that $\widetilde{X}$ contains an element of length $1$.

Let $w=(01)^\infty \in D.$
For each $n,m\in\mathbb N$, define
$x_n^m:=(01^n)^m w.$
Then $
f(x_n^{m+1})=x_n^m$, $m\ge 0$,
and hence $x_n:=x_n^0x_n^1x_n^2\cdots \in \widetilde{X}$
is an element of infinite length. Moreover, the sequence $(x_n^0)_{n\in\mathbb N}$ is constant equal to $w$, and therefore converges to $w$ in $D$. On the other hand,
\[x_n^1=(01^n)w \longrightarrow 01^\infty
\quad\text{in } X,\]
which means that $(x_n^1)_{n\in\mathbb N}$ converges to $\infty$ in $D^\infty$. It follows from Corollary~\ref{convergencia em W0} that $(x_n)_{n\in\mathbb N}$ converges to $w$ in $D_0$. Therefore $w\in \widetilde{X}$.

Since $\ell(w)=1$, Proposition~\ref{prop:shift-DR-correct} yields $\vec{0}\in E$. Consequently, Theorem~\ref{thm:transfer-shadowing} applies in this example as well.

\bibliographystyle{abbrv}
\bibliography{ref}

@article{DGS,
title = {Shadowing, finite order shifts and ultrametric spaces},
journal = {Advances in Mathematics},
volume = {385},
pages = {107760},
year = {2021},
issn = {0001-8708},
doi = {https://doi.org/10.1016/j.aim.2021.107760},
url = {https://www.sciencedirect.com/science/article/pii/S0001870821001997},
author = {Udayan B. Darji and Daniel Gon\c{c}alves and Marcelo Sobottka},
}

@article{ArmstrongBrixCarlsenEilers2023,
  author  = {Becky Armstrong and Kevin Aguyar Brix and Toke Meier Carlsen and S{\o}ren Eilers},
  title   = {Conjugacy of local homeomorphisms via groupoids and {C}*-algebras},
  journal = {Ergodic Theory and Dynamical Systems},
  volume  = {43},
  number  = {8},
  pages   = {2516--2537},
  year    = {2023},
  doi     = {10.1017/etds.2022.50}
}

@article{AraClaramunt2024,
  author  = {Pere Ara and Joan Claramunt},
  title   = {A correspondence between surjective local homeomorphisms and a family of separated graphs},
  journal = {Discrete and Continuous Dynamical Systems},
  volume  = {44},
  number  = {5},
  pages   = {1178--1266},
  year    = {2024},
  doi     = {10.3934/dcds.2023143}
}

@article{GoodMeddaugh2020,
  author    = {Chris Good and Jonathan Meddaugh},
  title     = {Shifts of finite type as fundamental objects in the theory of shadowing},
  journal   = {Inventiones mathematicae},
  volume    = {220},
  number    = {3},
  pages     = {715--736},
  year      = {2020},
  doi       = {10.1007/s00222-019-00936-8},
  url       = {https://doi.org/10.1007/s00222-019-00936-8}
}

@article{GU,
title = {Shadowing for Local Homeomorphisms, with Applications to Edge Shift Spaces of Infinite Graphs},
journal = {Journal of Dynamics and Differential Equations},
year = {2024},
issn = {1572-9222},
doi = {https://doi.org/10.1016/10.1007/s10884-023-10342-7},
url = {https://doi.org/10.1007/s10884-023-10342-7},
author = {Gon\c{c}alves, D. and Uggioni, B.B.},
}

@article{webster2014,
  author  = {Samuel B. G. Webster},
  title   = {The path space of a directed graph},
  journal = {Proceedings of the American Mathematical Society},
  volume  = {142},
  number  = {1},
  pages   = {213--225},
  year    = {2014},
  doi     = {10.1090/S0002-9939-2013-11755-7}
}

@article{webster2011,
  author  = {Samuel B. G. Webster},
  title   = {The path space of a higher-rank graph},
  journal = {Studia Mathematica},
  volume  = {204},
  number  = {2},
  pages   = {155--185},
  year    = {2011},
  doi     = {10.4064/sm204-2-4}
}

@article{deaconu1995,
  author  = {Valentin Deaconu},
  title   = {Groupoids associated with endomorphisms},
  journal = {Transactions of the American Mathematical Society},
  volume  = {347},
  number  = {5},
  pages   = {1779--1786},
  year    = {1995},
  doi     = {10.1090/S0002-9947-1995-1233967-5}
}

@article{chenli1993,
  author  = {Liang Chen and Shi Hai Li},
  title   = {Shadowing property for inverse limit spaces},
  journal = {Proceedings of the American Mathematical Society},
  volume  = {115},
  number  = {2},
  pages   = {573--580},
  year    = {1992},
  doi     = {10.1090/S0002-9939-1992-1097338-X}
}

@article{GoncalvesRoyerTasca2024,
  author  = {D. Gon{\c{c}}alves and D. Royer and F. Tasca},
  title   = {Entropy of local homeomorphisms with applications to infinite alphabet shift spaces},
  journal = {International Mathematics Research Notices},
  year    = {2024},
  volume  = {2024},
  number  = {6},
  pages   = {4842--4892}
}

@article {pat,
    AUTHOR = {Paterson, Alan L. T. and Welch, Amy E.},
     TITLE = {Tychonoff's theorem for locally compact spaces and an
              elementary approach to the topology of path spaces},
   JOURNAL = {Proceedings of the American Mathematical Society},
  FJOURNAL = {Proceedings of the American Mathematical Society},
    VOLUME = {133},
      YEAR = {2005},
    NUMBER = {9},
     PAGES = {2761--2770},
      ISSN = {0002-9939,1088-6826},
   MRCLASS = {54B10 (05C10 22A22 54B15)},
  MRNUMBER = {2146225},
MRREVIEWER = {Douglass\ L.\ Grant},
       DOI = {10.1090/S0002-9939-05-08030-5},
       URL = {https://doi.org/10.1090/S0002-9939-05-08030-5},
}

@book{OTW14,
 author = {Ott, William and Tomforde, Mark and Willis, Paulette N.},
 title = {One-sided shift spaces over infinite alphabets},
 fseries = {NYJM Monographs},
 series = {NYJM Monogr.},
 volume = {5},
 year = {2014},
 publisher = {Albany, NY: State University of New York, University at Albany},
 language = {English},
 zbMATH = {6274317},
 Zbl = {1284.37001}
}

@book{munkres2edition,
 author = {Munkres, James R.},
 title = {Topology.},
 edition = {2nd ed.},
 isbn = {0-13-181629-2},
 year = {2000},
 publisher = {Upper Saddle River, NJ: Prentice Hall},
 language = {English},
 zbMATH = {1461253},
 Zbl = {0951.54001}
}

\end{document}